\documentclass[a4paper,12pt]{amsart}

\usepackage[latin1]{inputenc}
\usepackage{amsfonts}
\usepackage{amsmath}
\usepackage{amssymb}
\usepackage{amsthm}
\usepackage[english]{babel}
\usepackage{graphicx}
\usepackage{mathdots}
\usepackage{xy}
\usepackage{hyperref, url}
\usepackage{xfrac}

\input xy
\xyoption{all}

\newtheorem{thm}{Theorem}[section]
\newtheorem{lemma}[thm]{Lemma}
\newtheorem{prop}[thm]{Proposition}
\newtheorem{cor}[thm]{Corollary}
\newtheorem{defn}[thm]{Definition}

\theoremstyle{remark}
\newtheorem{ex}[thm]{Example}
\newtheorem{rmk}[thm]{Remark}

\numberwithin{figure}{section}
\numberwithin{equation}{section}

\newcommand{\N}{\mathbb{N}}
\newcommand{\Z}{\mathbb{Z}}
\newcommand{\Q}{\mathbb{Q}}
\newcommand{\R}{\mathbb{R}}

\newcommand{\F}{\mathbb{F}}
\newcommand{\T}{\mathbb{T}}
\newcommand{\HF}{\widehat{HF}}
\newcommand{\CF}{\widehat{CF}}
\newcommand{\HFK}{\widehat{HFK}}
\newcommand{\CFK}{\widehat{CFK}}

\newcommand{\de}{\partial}
\newcommand{\ba}{\boldsymbol{\alpha}}
\newcommand{\bb}{\boldsymbol{\beta}}
\newcommand{\bc}{\boldsymbol{\gamma}}

\newcommand{\x}{\mathbf{x}}
\newcommand{\y}{\mathbf{y}}

\newcommand{\bfTheta}{\boldsymbol{\Theta}}
\renewcommand{\epsilon}{\varepsilon}
\newcommand{\spin}{{\rm Spin}^c}
\newcommand{\s}{\mathfrak{s}}
\renewcommand{\L}{\mathcal{L}}
\renewcommand{\H}{\mathcal{H}}
\newcommand{\D}{\mathcal{D}}

\newcommand{\Lhat}{\widehat{\mathcal{L}}}

\newcommand{\ltilde}{\widetilde}
\newcommand{\Sym}{{\rm Sym}}

\newcommand{\SFH}{\underrightarrow{SFH}}
\newcommand{\EH}{\underrightarrow{EH}}
\newcommand{\A}{\mathbf{A}}

\newcommand{\deff}{\textbf}

\title{Comparing invariants of Legendrian knots}
\author{Marco Golla}
\date{}

\begin{document}

\maketitle

\begin{abstract}
We prove the equivalence of the invariants $EH(L)$ and $\L^-(\pm L)$ for oriented Legendrian knots $L$ in the 3-sphere equipped with the standard contact structure, partially extending a previous result by Stipsicz and V\'ertesi. In the course of the proof we relate the sutured Floer homology groups associated with a knot complement $S^3\setminus K$ and the knot Floer homology of $(S^3,K)$ and define intermediate Legendrian invariants.
\end{abstract}

\section{Introduction}

In recent years, many Floer-theoretic invariants for Legendrian knots have been introduced: in 2008, Ozsv\'ath, Szab\'o and Thurston \cite{OST} used grid diagrams to define two invariants $\widehat{\lambda}_\pm(L)$ and $\lambda_{\pm}(L)$ of oriented Legendrian knots $L$ in $(S^3,\xi_{\rm st})$, taking values in a combinatorial version of knot Floer homology. Shortly afterwards, Lisca, Ozsv\'ath, Stipsicz and Szab\'o \cite{LOSS} used open books to construct two other invariants of oriented nullhomologous Legendrian knots $L$, called $\Lhat(L)$ and $\L^-(L)$, taking values in the original version knot Floer homology.

In 2006, Juh\'asz defined a version of Heegaard Floer homology for manifolds with ``marked'' boundary, which he called sutured Floer homology \cite{Ju}. Honda, Kazez and Mati\'c soon constructed invariants for contact manifolds with convex boundary, taking values in a sutured Floer cohomology group \cite{HKM1}: the key feature of their invariant (and of sutured Floer homology) is its behaviour with respects to gluing manifolds along their (compatible) boundaries \cite{HKM2}.

In this context, to every Legendrian knot $L$ in a contact three-manifold one can associate a contact manifold with convex boundary, and therefore a contact invariant $EH(L)$ living in some sutured Floer homology group. Some natural questions arise at this point: is there any relation between $EH(L)$ and the $\L$ invariants? If so, what is this relation exactly?

Late in 2008, a first answer to these questions was given by Stipsicz and V\'ertesi, who explained how $EH(L)$ determines $\Lhat(L)$ \cite{SV}; recently, Baldwin, Vela--Vick and V\'ertesi were able to prove the equivalence of the combinatorial invariants $\lambda$ and the LOSS invariants $\L$ \cite{BVV}.

Our main result is the following ($-L$ means $L$ with the reversed orientation).

\begin{thm}\label{mainthm}
For two oriented, topologically isotopic Legendrian knots $L_0,L_1$ in $(S^3,\xi_{\rm st})$, the following are equivalent:
\begin{itemize}
\item[(i)] $EH(L_0) = EH(L_1)$;
\item[(ii)] $\L^-(L_0) = \L^-(L_1)$ and $\L^-(-L_0) = \L^-(-L_1)$.
\end{itemize}
\end{thm}

The same result has been obtained, in greater generality, by Etnyre, Vela--Vick and Zarev \cite{EVZ}. In fact, using the same techniques together with a generalisation of \cite[Theorem 11.35]{LOT}, one can prove the generalisation of Theorem \ref{mainthm} to Legendrian knots in arbitrary contact 3-manifolds $(Y,\xi)$ such that $c(Y,\xi)\neq 0$, and it is always the case that $EH(L)$ determines $\L^-(\pm L)$.

\vskip 0,2 cm

{\bf Organisation}. This paper is organised as follows: we first review the setting we're working in, giving a brief introduction to sutured Floer homology in Section \ref{SFH+gluing} and the $EH$ invariants in Section \ref{HFKstab}. Then we analyse in some detail the groups and the maps we are dealing with, in Section \ref{newLeg}. In Section \ref{last} the relation between various sutured Floer homology associated to a knot complement and $HFK^-$ are explained; this will lead to the proof of the equivalence of the two invariants $EH$ and $\L^-$ in the last section.

\vskip 0,2 cm

{\bf Acknowledgments}. I'm very grateful to my supervisor, Jake Rasmussen, for suggesting me the problem, for many helpful discussions, and for his support. I want to thank Paolo Lisca, Olga Plamenevskaya, Andr\'as Stipsicz and David Shea Vela-Vick for interesting conversations, and the referees for helpful comments and suggestions. Part of this work has been done while I was visiting the Simons Center for Geometry and Physics: I acknowledge their support. The author has been supported by the ERC grant LTDBUD.

\section{Sutured Floer homology and gluing maps}\label{SFH+gluing}

\subsection{Sutured manifolds}\label{sutured}

The definition of balanced sutured manifold is due to Juh\'asz \cite{Ju}.

\begin{defn}
A \textbf{balanced sutured manifold}, is a pair $(M,\Gamma)$ where $M$ is an oriented 3-manifold with nonempty boundary $\de M$, and $\Gamma$ is a family of oriented curves in $\de M$ that satifies:
\begin{itemize}\itemsep-1.5pt
 \item $\Gamma$ intersects each component of $\de M$;

 \item $\Gamma$ disconnects $\de M$ into $R_+$ and $R_-$, with $\pm\Gamma = \de R_\pm$ (as oriented manifolds);

 \item $\chi (R_+) = \chi (R_-)$.
\end{itemize}
\end{defn}

\begin{rmk}
The condition $\chi(R_+) = \chi(R_-)$ is called the \emph{balancing} condition. Since this is the only kind of sutured manifolds we're dealing with, we prefer to just drop the adjective `balanced'.
\end{rmk}

\begin{ex}\label{cont_sut_ex1}
Any $M$ oriented 3-manifold with $S^2$-boundary, can be turned into a sutured manifold $(M,\{\gamma\})$ by choosing any simple closed curve $\gamma$ in $\de M$. We'll often write $M=Y(1)$, where $Y = M\cup_\de D^3$ is the ``simplest'' closed 3-manifold containing $M$.

For every integer $f$, we have a sutured manifold $S^3_{K,f}$ given by pairs $(S^3\setminus N(K), \{\gamma_f,-\gamma_f\})$, where $\gamma_f$ is an oriented curve on the boundary torus $\de N(K)$ of an open small neighbourhood $N(K)$ of $K$. The slope of $\gamma_f$ is $\lambda_S + f\cdot \mu$, and $-\gamma_f$ is a parallel push-off of $\gamma_f$, with the opposite orientation. 
Here $\lambda_S$ denotes the Seifert longitude of $K$. We'll use the shorthand $\Gamma_f$ for $\{\gamma_f,-\gamma_f\}$.
\end{ex}

\begin{ex}\label{cont_sut_ex2}
To any Legendrian knot $L\subset (Y,\xi)$ in an arbitrary 3-manifold $Y$ one can associate in a natural way a sutured manifold, that we'll denote with $Y_L$, constructed as follows: there's a standard open Legendrian neighbourhood $\nu(L)$ for $L$, whose complement has convex boundary. The dividing set $\Gamma_L$ on the boundary consists of two parallel oppositely oriented curves parallel to the contact framing of $L$. The manifold $Y_L$ is then defined as the pair $(Y\setminus \nu(L), \Gamma_L)$. In the case we're mainly interested in, where $Y=S^3$ and $L$ is of topological type $K$, we have $S^3_L = S^3_{K, tb(L)}$. More generally, the same identification $\{{\rm framings}\}\leftrightarrow \Z$ can be made canonical whenever $K$ is nullhomologous in $Y$ and $H_2(Y)=0$, and we then have $Y_L = Y_{K,tb(L)}$.

We'll often use $Y_L$ also to denote the contact manifold with convex boundary $(Y\setminus \nu(L),\xi|_{Y\setminus \nu(L)})$, without creating any confusion.

\end{ex}

There's a decomposition/classification theorem for sutured manifolds, completely analogous to the Heegaard decomposition/Reidemeister-Singer theorem for closed three-manifolds. Consider a compact surface $\Sigma$ with boundary together with two collections of pairwise disjoint simple closed curves $\ba,\bb\subset \Sigma$, such that each collection is a linearly independent set in $H_1(\Sigma; \Z)$; suppose moreover that $|\ba|=|\bb|$. We can build a balanced sutured manifold out of this data as follows: take $\Sigma\times[0,1]$, glue a 2-handle on $\Sigma\times\{0\}$ for each $\alpha$-curve, and a 2-handle on $\Sigma\times\{1\}$ for each $\beta$-curve, and let $M$ be the manifold obtained after smoothing corners; declare $\Gamma = \de\Sigma\times\{1/2\}$. The pair $(M,\Gamma)$ is a balanced sutured manifold, and $(\Sigma,\ba,\bb)$ is called a \deff{(sutured) Heegaard diagram} of $(M,\Gamma)$.

\begin{thm}[\cite{Ju}]
Every balanced sutured manifold admits a Heegaard diagram, and every two such diagrams become diffeomorphic after a finite number of isotopies of the curves, handleslides and stabilisations taking place in the interior of the Heegaard surface.
\end{thm}

\subsection{The Floer homology packages}

This is meant to be just a recollection of facts about the Floer homology theories we'll be working with. The standard references for the material in this subsection are \cite{OSHF, OSPA,Lip} for the Heegaard Floer part, and \cite{Ju} for the sutured Floer part.

In order to avoid sign issues, we'll work with $\F=\F_2$ coefficients.

Consider a pointed Heegaard diagram $\H=(\Sigma_g,\ba,\bb,z)$ representing a three-manifold $Y$, and form two Heegaard Floer complexes $\widehat{CF}(Y)$ and $CF^-(Y)$: the underlying modules are freely generated over $\F$ and $\F[U]$ by $g$-tuples of intersection points in $\bigcup_{i,j}(\alpha_i\cap\beta_j)$, so that there's exactly one point on each curve in $\ba\cup\bb$.

The differentials $\widehat{\de}, \de^-$ are harder to define, and count certain pseudo-holomorphic discs in a symmetric product $\Sym^g(\Sigma_g)$, or maps from Riemann surfaces with boundary in $\Sigma_g\times\R\times [0,1]$, with the appropriate boundary conditions. The homology groups $\HF(Y)=H_*(\CF(Y),\widehat{\de})$ and $HF^-(Y)=H_*(CF^-(Y),\de^-)$ so defined are called \deff{Heegaard Floer homologies} of $Y$, and are independent of the (many) choices made along the way \cite{OSHF}.

Sutured Floer homology is a variant of this construction for sutured manifolds $(M,\Gamma)$. The starting point is a sutured Heegaard diagram $\H=(\Sigma,\ba,\bb)$ for $(M,\Gamma)$. We form a complex $SFC(M,\Gamma)$ in the same way, generated over $\F$ by $d$-tuples of intersection points as above, where $d=|\ba|=|\bb|$. The differential $\de$ is defined by counting pseudo-holomorphic discs in $\Sym^d(\Sigma)$ or maps from Riemann surfaces to $\Sigma\times\R\times[0,1]$, again with the appropriate boundary conditions.

The homology $SFH(M,\Gamma)=H_*(SFC(M,\Gamma),\de)$ is called the \deff{sutured Floer homology} of $(M,\Gamma)$, and is shown to be independent of all the choices made \cite{Ju}. It naturally corresponds to a `hat' theory.

\begin{prop}[\cite{Ju}]\label{HFvsSFH}
For a closed 3-manifold $Y$, $\HF (Y) = SFH(Y(1))$.

For a knot $K$ in a closed 3-manifold $Y$, $\widehat{HFK}(Y,K) = SFH(Y_{K,m})$, where $m$ is the meridian for $K$ in $Y$.
\end{prop}

\subsection{Floer-theoretic contact invariants}

The first contact invariant to be defined in Heegaard Floer homology was Ozsv\'ath and Szab\'o's $c$ \cite{OScontact}. We sketch here the construction of the contact class $EH$ \cite{HKM1}, and we will relate it to $c$ below.

\begin{defn}
A \deff{partial open book} is a triple $(S,P,h)$ where $S$ is a compact open surface, $P$ is a proper subsurface of $S$ which is a union of 1-handles attached to $S\setminus P$ and $h:P\to S$ is an embedding that pointwise fixes a neighborhood of $\de P \cap \de S$.
\end{defn}

We can build a contact manifold with convex boundary out of these data in a fashion similar to the usual open books: instead of considering a mapping torus, though, we glue two asymmetric halves, quotienting the disjoint union $S\times[0,1/2]\coprod P\times [1/2,1]$ by the relations $(x,t)\sim (x,t')$ for $x\in \de S$, $(y,1/2)\sim (y,1/2)$, $(h(y),1/2)\sim (y,1)$ for $y\in P$. The contact structure is uniquely determined if we require -- as we do -- tightness and prescribed sutures on each half $S\times[0,1/2]/\mathord{\sim}$ and $P\times[1/2,1]/\mathord{\sim}$ (see \cite{Ho} for details). Moreover, to any contact manifold with convex boundary we can associate a partial open book, unique up to Giroux stabilisations.

We can build a balanced diagram out of a partial open book. The Heegaard surface $\Sigma$ is obtained by gluing $P$ to $-S$ along the common boundary.

\begin{defn}\label{basis}
A \deff{basis} for $(S,P)$ is a set $\mathbf{a} = \{a_1,\dots,a_k\}$ of arcs properly embedded in $(P,\de P \cap \de S)$ whose homology classes generate $H_1(P,\de P \cap \de S)$.
\end{defn}

Given a basis as above, we produce a set $\mathbf{b} = \{b_1,\dots,b_k\}$ of curves using a Hamiltonian vector field on $P$: we require that under this perturbation the endpoints of $a_i$ move in the direction of $\de P$, and that each $a_i$ intersects $b_i$ in a single point $x_i$, and is disjoint from all the other $b_j$'s.

Finally define the two sets of attaching curves: $\ba = \{\alpha_i\}$ and $\bb = \{\beta_i\}$, where $\alpha_i = a_i \cup -a_i$ and $\beta_i = h(b_i) \cup -b_i$: the sutured manifold associated to $(\Sigma,\ba,\bb)$ is $(M,\Gamma)$. We call $\x(S,P,h)$ the generator $\{x_1,\dots,x_k\}$ in $SFC(\Sigma,\bb,\ba)$ supported inside $P$.

\begin{thm}[\cite{HKM1}]
The chain $\x(S,P,h)\in SFC(\Sigma,\bb,\ba)$ is a cycle, and its class in $SFH(-M,-\Gamma)$ is an invariant of the contact manifold $(M,\xi)$ defined by the partial open book $(S,P,h)$.
\end{thm}

\begin{defn}
$EH(M,\xi)$ is the class $[\x(S,P,h)] \in SFH(-M,-\Gamma)$ for some partial open book $(S,P,h)$ supporting $(M,\xi)$.
\end{defn}

The type of invariants that we're going to deal with are either invariants of (complements of) Legendrian knots or invariants coming from contact structures on closed manifolds: this allows us to consider only sutured manifolds with sphere/torus boundary and one/two sutures, as described in Examples \ref{cont_sut_ex1} and \ref{cont_sut_ex2}.

Consider a closed contact manifold $(Y,\xi)$, and let $B\subset Y$ be a small, closed Darboux ball with convex boundary. Then consider the manifold $(Y(1),\xi(1))$ where $Y(1)$ is obtained from $Y$ by removing the interior of $B$, and $\xi(1)$ is $\xi|_{Y(1)}$.

\begin{prop}[\cite{HKM1}]
There is an isomorphism of graded complexes from $\HF(Y)$ to $SFH(Y(1))$ that maps the Ozsv\'ath-Szab\'o contact invariant $c(Y,\xi)$ to the Honda-Kazez-Mati\'c class $EH(Y(1),\xi(1))$.
\end{prop}

Suppose now that $L\subset Y$ is a Legendrian knot with respect to a contact structure $\xi$: the contact manifold $Y_L$ defined in Example \ref{cont_sut_ex2} determines a contact invariant $EH(Y_L) \in SFH(-Y_L)$. We'll denote this invariant by $EH(L)$, considering it as an invariant of the Legendrian isotopy class of $L$ rather than of its complement.

\subsection{Gluing maps}

In their paper \cite{HKM2}, Honda, Kazez and Mati\'c define maps associated to the gluing of a contact manifold to another one along some of the boundary components, and show that these maps preserve their $EH$ invariant. Consider two sutured manifolds $(M,\Gamma) \subset (M',\Gamma')$, where $M$ is embedded in ${\rm Int}(M')$; let $\xi$ be a contact structure on $N:=M'\setminus {\rm Int}(M)$ such that $\de N$ is $\xi$-convex and has dividing curves $\Gamma \cup \Gamma'$. For simplicity, and since this will be the only case we need, we'll restrict to the case when each connected component of $N$ intersects $\de M'$ (\emph{i.e.} gluing $N$ to $M$ doesn't kill any boundary component).

\begin{thm}\label{EHmap}
The contact structure $\xi$ on $N$ induces a \deff{gluing map} $\Phi_\xi$, that is a linear map $\Phi_\xi: SFH(-M,-\Gamma)\to SFH(-M',-\Gamma')$. If $\xi_M$ is a contact structure on $M$ such that $\de M$ is $\xi_M$-convex with dividing curves $\Gamma$, then $\Phi_{\xi}(EH(M,\xi_M)) = EH(M',\xi_M\cup \xi)$.
\end{thm}

This theorem has interesting consequences, even in simple cases:

\begin{cor}\label{nonvanishing}
If $(M,\Gamma)$ embeds in a Stein fillable contact manifold $(Y,\xi)$, and $\de M$ is $\xi$-convex, divided by $\Gamma$, then $EH(M,\xi|_M)$ is not trivial.
\end{cor}

There's also a naturality statement, concerning the composition of two gluing maps: suppose that we have three sutured manifolds $(M,\Gamma)\subset(M',\Gamma')\subset(M'',\Gamma'')$ as at the beginning of the section, and suppose that $\xi$ and $\xi'$ are contact structures on $M'\setminus{\rm Int}(M)$ and $M''\setminus{\rm Int}(M')$ respectively, that induce sutures $\Gamma$, $\Gamma'$ and $\Gamma''$ on $\de M$, $\de M'$ and $\de M''$ respectively.

\begin{thm}\label{associativityphi}
If $\xi$ and $\xi'$ are as above, then $\Phi_{\xi\cup\xi'} = \Phi_{\xi'}\circ\Phi_{\xi}$.
\end{thm}

Much of our interest will be devoted to stabilisations of Legendrian knots and associated maps, whose discussion will occupy Subsection \ref{substabmaps}: we give a brief summary of the contact side of their story here.

Let's start with a definition, due to Honda \cite{Ho}:
\begin{defn}\label{basic_slices}
Let $\eta$ be a tight contact structure on $T^2\times I$ with two dividing curves on each boundary component: call $\gamma_i$, $-\gamma_i$ the homology class of the two dividing curves on $T^2 \times\{i\}$, and let $s_i\in \Q\cup\{\infty\}$ be their slope. $(T^2 \times I,\eta)$ is a \deff{basic slice} if it is of the form above, and also satisfies the following three conditions:
\begin{itemize}\itemsep-1pt
\item $\{\gamma_0, \gamma_1\}$ is a basis for $H_1(T^2)$;
\item $\xi$ is \deff{minimally twisting}, \emph{i.e.} if $T_t = T\times \{t\}$ is convex, the slope of the dividing curves on $T_t$ belongs to $[s_0, s_1]$ (where we assume that if $s_0>s_1$ the interval $[s_0,s_1]$ is $[-\infty,s_1]\cup[s_0,\infty]$);
\end{itemize}
\end{defn}

Honda proved the following:

\begin{prop}[\cite{Ho}]
For every integer $t$ there exist exactly two basic slices $(T^2\times I, \xi_j)$ (for $j=1,2$) with boundary slopes $t/1$ and $(t-1)/1$. The sutured complement of a stabilisation $L'$ of $L$ is gotten by attaching one of the two basic slices to $Y_L$, where the trivialization of $T^2$ is given by identifying the slopes $0/1$ and $t/1$ with a meridian $\mu$ and the contact framing $c$ for $L$, respectively.
\end{prop}

These two different layers correspond to the positive and negative stabilisation of $L$, once we've chosen an orientation for the knot; reversing the orientation swaps the labelling signs. Since we'll be considering oriented Legendrian knots, we can label the two slices with a sign.

\begin{defn}\label{stabmaps_def}
We call \deff{stabilisation maps} the gluing maps associated to the attachment of a stabilisation basic slice: these will be denoted with $\sigma_\pm$.
\end{defn}

\begin{rmk}
As it happens for the Stipsicz-V\'ertesi map \cite{SV}, these basic slice attachments correspond to single bypass attachments, too.
\end{rmk}

\section{A few facts on $SFH(S^3_{K,n})$ and $\sigma_{\pm}$}\label{HFKstab}

Given a topological knot $K$ in $S^3$, denote with $S^3_m(K)$ the manifold obtained by (topological) $m$-surgery along $K$, and let $\ltilde{K}$ be the dual knot in $S^3_m(K)$, that is the core of the solid torus we glue back in. Notice that an orientation on $K$ induces an orientation of $\ltilde{K}$, by imposing that the intersection of the meridian $\mu_K$ of $K$ on the boundary of the knot complement has intersection number $+1$ with the meridian $\mu_{\ltilde K}$ of $\ltilde{K}$ on the same surface.

Fix a contact structure $\xi$ on $S^3$ and a Legendrian representative $L$  of $K$: we'll write $t$ for $tb(L)$. Since $t$ measures the difference between the contact and the Seifert framings of $L$, $S^3_{t}(K)_{\ltilde{K},\infty}$ and $S^3_L$ are sutured diffeomorphic: in particular, $EH(L)$ lives in $SFH(-S^3_{t}(K)_{\ltilde{K},\infty}) = \HFK(-S^3_{t}(K),\ltilde{K})$, the identification depending on the choice of an orientation for $K$ (or $\ltilde{K}$).

We will often write $\CFK(Y,K)$ to denote any chain complex computing $\HFK(Y,K)$ that comes from a Heegaard diagram, even though the complex itself depends on the choice of the diagram.

\subsection{Gradings and concordance invariants}

The groups $\HFK(S^3,K)$ and $\HFK(-S^3_{m}(K),\ltilde{K})$ come with a grading, that we call the \deff{Alexander grading}. A Seifert surface $F\subset S^3$ for $K$ gives a relative homology class
\[
[F,\de F]\in H_2(S^3\setminus N(K), \de N(K)) = H_2(S^3_m(K)\setminus N(\ltilde{K}), \de N(\ltilde{K})).\]

Given a generator $\x\in\CFK(S^3,K)$, there's an induced relative $\spin$ structure $\s(\x)$ in $\underline{\rm Spin}^c(S^3,K)$ \cite[Equation 2]{HP}, and the Alexander grading of $\x$ is defined as
\[
A(\x) = \frac12\langle c_1(\s(\x))-PD([\mu_K]), [F,\de F]\rangle,
\]
where $PD$ denotes Poincar\'e duality.

Likewise, given a generator $\x\in\CFK(-S^3_{m}(K),\ltilde{K})$, there's an induced relative $\spin$ structure $\s(\x)\in\underline{\rm Spin}^c(S^3_m(K),\ltilde{K})$, and we can define $A(\x)$ as
\begin{equation}\label{defAlex}
A(\x) = \frac12\langle c_1(\s(\x))-PD([\mu_{\ltilde K}]), [F,\de F]\rangle.
\end{equation}

We now turn to recalling the definition of $\tau(K)$, due to Ozsv\'ath and Szab\'o \cite{OStau}.

Recall that the Alexander grading induces a filtration on the knot Floer chain complex $(\CFK(S^3,K), \de)$, where the differential $\de$ ignores the presence of the second basepoint, that is $H_*(\CFK(S^3,K), \de) = \HF(S^3)$. In particular, every sublevel $\CFK(S^3,K)_{A\le s}$ is preserved by $\de$, and we can take its homology.

\begin{defn}
$\tau(K)$ is the smallest integer $s$ such that the inclusion of the $s$-th filtration sublevel induces a nontrivial map \[H_*(\CFK(S^3,K)_{A\le s}, \de)\longrightarrow \HF(S^3) = \F.\]
\end{defn}

This invariant turns out to provide a powerful lower bound for the slice genus of $K$, in the sense that $|\tau(K)|\le g_*(K)$ \cite{OStau}. One of the properties it enjoys, and that we'll need, is that $\tau(\overline{K}) = -\tau(K)$ for every $K$.

\subsection{Modules}\label{modules}

We now turn our attention back to $\HFK(-S^3_t(K),\ltilde{K}) \simeq SFH(-S^3_{K,t})$. Recall that this is a $\F$-vector space on which the $A$ defines a grading.

The group $\CFK(S^3,K)$ is a graded vector space that comes with two differentials, $\de_K$ and $\de$, such that the complex $(\CFK(S^3,K),\de)$ has homology $\HF(S^3) = \F$, while the complex $(\CFK(S^3,K),\de_K)$ is the associated graded object with respect to the Alexander filtration. By definition $\HFK(S^3,K)$ is the homology of this latter complex; as such, it inherits an Alexander grading that we call $A$.

Let's call $d = \dim\HFK(S^3,K)$, and fix a basis $\mathcal{B} = \{\eta_i,\eta'_j \mid 0\le i < d\}$ of $\CFK(S^3,K)$ such that the set $\{\eta^{top}_i, (\eta'_j)^{top}\}$ of the highest nontrivial Alexander-homogeneous components of the $\eta_i$'s and $\eta'_j$'s is still a basis for $\CFK(S^3,K)$, and the following relations hold (see \cite[Section 11.5]{LOT}):
\[
\begin{array}{lll}
\de \eta_0 = 0 &\quad & \de_K \eta_0 = 0 \\ \de \eta_{2i-1} = \eta_{2i} &\quad & \de_K \eta_i =  0 \\
\de \eta_{2j-1}'  = \eta_{2j}' &\quad & \de_K \eta'_{2j-1} = \eta'_{2j}.
\end{array}
\]
Observe that the set of homology classes of the $\eta_i$'s is a basis for $\HFK(S^3,K) = H_*(\CFK(S^3,K),\de_K)$. We'll write $A(\eta)$ for $A(\eta^{top})$. Finally, call $\delta(i) = A(\eta_{2i-1})-A(\eta_{2i})$; let's remark that by definition $A(\eta_0) = \tau:=\tau(K)$.

\begin{thm}[\cite{LOT}]\label{Heddenthm}
The \emph{homology} group $\HFK(-S^3_{m}(K),\ltilde{K})$ is an $\F$-vector space with basis $\{d_{i,j}, d_{i,j}^*, u_\ell\mid 1\le i \le k, 1\le j \le \delta(i), 1\le \ell \le |2\tau-m|\}$, where the generators satisfy $A(d_{i,j}) = A(\eta_{2i})-(j-1)-(m-1)/2 = -A(d_{i,j}^*)$ and $A(u_\ell) = \tau-(\ell-1)-(m-1)/2$.
\end{thm}

Generators with a $*$ are to be thought of as symmetric to the generators without it, and each family $\{d_{i,j}\}_j$ can be interpreted as representing the arrow $\eta_{2i-1}\stackrel{\de}{\mapsto} \eta_{2i}$ (notice that $i$ varies among \emph{positive} integers), counted with a multiplicity equalling its length (\emph{i.e.} the distance it covers in Alexander grading).

\begin{rmk}
Not any basis of $\HFK(-S^3_m(K),\ltilde{K})$ with the same degree properties works for our purposes: we're actually choosing a basis that's compatible with stabilisation maps, as we're going to see in Theorem \ref{stabmaps}.
\end{rmk}

\begin{defn}\label{defn3.3}
Call $S_+$ the subspace of $\HFK(-S^3_{m}(K),\ltilde{K})$ generated by $\{d_{i,j}\}$, and $S_-$ the one generated by $\{d_{i,j}^*\}$: the subspace $S=S_+ \oplus S_-$ is the \deff{stable complex}, and elements of $S$ are called \deff{stable elements}. The subspace spanned by $\{u_\ell\}$ is called the \deff{unstable complex} and will be denoted with $U_m$ (although the subscript will often be dropped), so that $\HFK(-S^3_{m}(K),\ltilde{K})$ decomposes as $S_+\oplus U_m \oplus S_-$.
\end{defn}

It's worth remarking that the decomposition given in the definition above \emph{does} depend on our choice of the basis: the three stable subspaces $S_\pm$ and $S$ are independent on this choice, but the unstable complex isn't; see also Remark \ref{torsion_rmk} below.

There's a good and handy pictorial description when $|m|$ is sufficiently large; we'll be mostly dealing with negative values of $m$, so let's call $m'=-m\gg 0$. Consider a direct sum $\ltilde{C} = \bigoplus_{i=1}^{m'} C_i$ of $m'$ copies of $C=\CFK(S^3,K)$, and (temporarily) denote by $\x_i$ the copy of the element $\x\in C$ in $C_i$. Endow $\ltilde{C}$ with a shifted Alexander grading:
\[
\ltilde{A}(\x_i) = \left\{\begin{array}{ll}
A(\x)-(i-1)-(m-1)/2 & \text{for } i\le m'/2\\
-A(\x)-(i-1)-(m-1)/2 & \text{for } i>m'/2
\end{array}\right.
\]
for each homogeneous $\x$ in $\CFK(S^3,K)$. We picture this situation by considering each copy of $C$ as a vertical tile of $2g(K)+1$ boxes -- each corresponding to a value for the Alexander grading, possibly containing no generators at all, or more than one generator -- and stacking the $m'$ copies of $C$ in staircase fashion, with $C_1$ as the top block and $C_{m'}$ as the bottom block. Notice that, by our grading convention, the copies in the bottom part of the picture are turned upside down: for example, if $\x^{\rm max}\in C$ has maximal Alexander degree $A(\x) = g(K)$, then $\x^{\rm max}_1$ lies in the top box of $C_1$, while $\x^{\rm max}_{m'}$ lies in the bottom box of $C_{m'}$. Likewise, an element $\x^\tau\in C$ has Alexander degree $A(\x)=\tau$, then $\x^\tau_1$ lies in the $(g(K)-\tau+1)$-th box from the top in $C_1$, and $\x_{m'}^\tau$ lies in the $(g(K)-\tau+1)$-th box from the bottom in $C_{m'}$.

Our construction is a variant of Hedden's construction: while in general our chain complex for $\HFK(S^3_m(K),\ltilde{K})$ differs in from his complex in the region with intermediate Alexander grading, the resulting homologies nevertheless agree.

The situation is depicted in Figure \ref{Heddenfigure}: in this concrete example we have $g(K)=2$ and $\tau(K)=-1$; accordingly, there are $2g(K)+1=5$ boxes in each vertical column and $\x^\tau_1$ lies in the fourth box from the top in $C_1$.

\begin{figure}
\begin{center}
\includegraphics[scale=0.6]{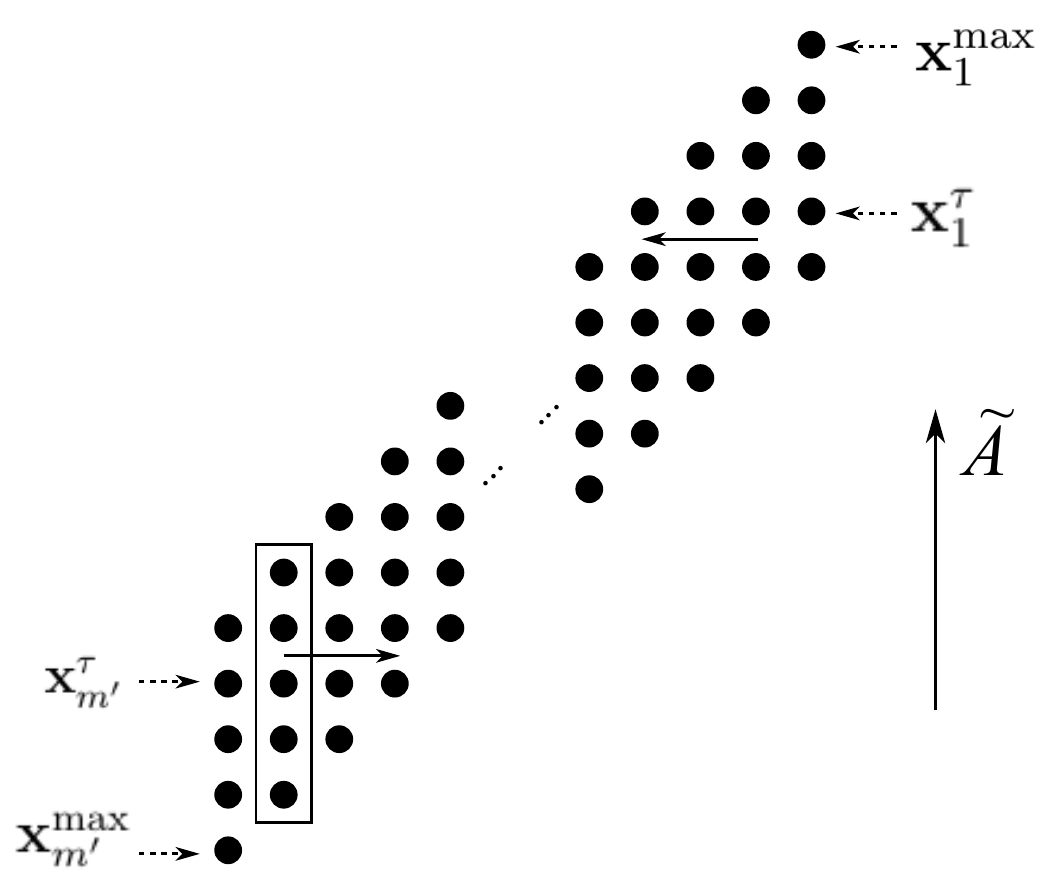}
\end{center}
\caption{We represent here the top (on the right) and bottom (on the left) parts of $\HFK(S^3_m(K),\ltilde{K})$ for  $m\ll 0$. Each vertical tile is a copy of $\CFK(S^3,K)$, and the arrows show the direction of the differentials.}\label{Heddenfigure}
\end{figure}

Now define a differential $\ltilde{\de}$ on $\ltilde{C}$ in the following way:
\[\ltilde{\de}:\left\{\begin{array}{ll}
 (\eta_0)_i \mapsto 0 & \text{for small and large } i\\
 (\eta_{2j-1})_i \mapsto (\eta_{2j})_{i+\delta(j)}\mapsto 0 & \text{for small } i\\
 (\eta_{2j-1})_i \mapsto (\eta_{2j})_{i-\delta(j)}\mapsto 0 & \text{for large } i\\
 (\eta'_{2j-1})_i \mapsto (\eta'_{2j})_i\mapsto 0 & \text{for every } i\\
  \end{array}\right..
\]
We extend the differential to be any map $\ltilde{\de}$ such that the level $\{A = j\}$ is a subcomplex for every $j$, whose homology is $\F$ for intermediate values of $j$ (this is possible since $\{A = j\}$ has odd rank for every intermediate value of $j$).

We're now going to analyse what happens on the top and bottom part of the complex (\emph{i.e.} when $i$ is small or large, in what follows), when we take the homology.

Pairs $(\eta'_{2j-1})_i, (\eta'_{2j})_i$ cancel out in homology. The element $(\eta_{2j})_i$ is a cycle for each $i,j$, and it's a boundary only when $j>0$ and either $i>\delta(j)$ or $i< m'-\delta(j)$: so there are $2\delta(j)$ surviving copies of $\eta_{2j}$, in degrees $A(\eta_{2j})-k-(m-1)/2$ and $-A(\eta_{2j})+k+(m-1)/2$ for $k=0,\dots,\delta(j)-1$. We can declare $d_{i,j} = [(\eta^{top}_{2j})_i]$ and $d_{i,j}^* = [(\eta^{top}_{2j})_{m'-i}]$.

The element $(\eta_0)_i$ is a cycle for every $i$, and it's never canceled out, so it survives when taking homology. Given our grading convention, for \emph{small} values of $i$, $\ltilde{A}((\eta_0)_i) = A(\eta_0)-(i-1)-(m-1)/2 = \tau(K)-(i-1)-(m-1)/2$, and in particular we have a nonvanishing class $[(\eta^{top}_0)_i] = u_i$ in degrees $\tau(K)-(m-1)/2, \tau(K)-(m-1)/2-1, \dots$ On the other hand, when $i$ is \emph{large}, $[(\eta_0)_i]$ lies in degree $-\tau(K)-(i-1)-(m-1)/2$, and we get a nonvanishing class $[(\eta^{top}_0)_i] = u_{2\tau(K)+i+(m-1)/2}$ in degrees $-\tau(K)+(m-1)/2, -\tau(K)+(m-1)/2+1, \dots$

We also have a string of $\F$ summands in between, giving us a strip of unstable elements of length $2\tau(K)-m$, as in Theorem \ref{Heddenthm}.


\subsection{Stabilisation maps}\label{substabmaps}

We're going to study the action of the two stabilisation maps $\sigma_\pm$ of Definition \ref{stabmaps_def} on the sutured Floer homology groups $SFH(-S^3_L)$. It's worth stressing that these maps do not depend on the particular Legendrian representative, but only on its Thurston-Bennequin number: in fact, the topological type of $L$ determines the complement $S^3\setminus \nu(L)$ and $tb(L)$ determines the sutures on $\de\nu(L)$, hence the sutured manifold $S^3_L$ depends only on these data. A gluing map $\Phi_\xi: SFH(M,\Gamma) \to SFH(M',\Gamma')$ only depends on the contact structure $\xi$ on the layer and not on the contact structure on $(M,\Gamma)$ (in fact, no such contact structure is required in the definition of $\Phi_\xi$).

Notice that if $L$ is a Legendrian knot in $S^3$ with $tb(L)=n$, then, as a sutured manifold, $S^3_L$ is just $S^3_{K,n}$. Moreover, if $L'$ is a stabilisation of $L$, then $S^3_{L'}$ is isomorphic to $S^3_{K,n-1}$ as a sutured manifold.

Recall that we have two families (indexed by the integer $n$) of stabilisation maps, $\sigma_\pm: SFH(-S^3_{K,n})\to SFH(-S^3_{K,n-1})$, corresponding to the gluing of the negative and positive stabilisation layer: if the knot $K$ is oriented, these maps can be labelled as $\sigma_-$ or $\sigma_+$. With a slight abuse of notation, we're going to ignore the dependence of these maps on the framing.

\begin{rmk}\label{orientation_rmk}
Notice that orientation reversal of $L$ or $K$ isn't seen by the sutured groups nor by $EH(L)$, but it swaps the r\^oles of $\sigma_-$ and $\sigma_+$.
\end{rmk}

\begin{rmk}
Let's recall that for an \emph{oriented} Legendrian knot $L$ of topological type $K$ in $S^3$ the Bennequin inequality holds:
\[tb(L)+r(L) \le 2g(K)-1.\]
In \cite{Pl}, Plamenevskaya proved a sharper result:
\begin{equation}\label{olga}
tb(L)+r(L) \le 2\tau(K)-1.
\end{equation}

This last form of the Bennequin inequality, together with Theorem \ref{Heddenthm}, tells us that, whenever we're considering knots in the standard $S^3$, the unstable complex is never trivial in $SFH(-S^3_{K,n})$: more precisely we're always (strictly) below the threshold $2\tau:=2\tau(K)$, so that $2\tau-m$ is always positive; in particular, the dimension of the unstable complex is always positive and \emph{increases} under stabilisations. We'll state the theorem in its full generality anyway, even though this remark tells us we need just half of it when working in $(S^3,\xi_{\rm st})$.
\end{rmk}

The following theorem is proved in \cite[Section 3.4]{me}.

\begin{thm}\label{stabmaps}
The maps $\sigma_-, \sigma_+: SFH(-S^3_{K,n})\to SFH(-S^3_{K,n-1})$ act as follows:
\[\begin{array}{lcll}\sigma_-:\left\{\begin{array}{l}
d_{i,j}\mapsto d_{i,j}\\
u_\ell\mapsto u_\ell\\
d^*_{i,j}\mapsto d^*_{i,j+1}
  \end{array}\right. , & &
\sigma_+:\left\{\begin{array}{l}
d_{i,j}\mapsto d_{i,j+1}\\
u_\ell\mapsto u_{\ell+1}\\
d^*_{i,j}\mapsto d^*_{i,j}
  \end{array}\right. & {\rm for }\; n\le 2\tau;\\
  \\
\sigma_-:\left\{\begin{array}{l}
d_{i,j}\mapsto d_{i,j}\\
u_\ell\mapsto u_\ell\\
u_{n-2\tau}\mapsto 0 \\
d^*_{i,j}\mapsto d^*_{i,j+1}
  \end{array}\right. , & &
\sigma_+:\left\{\begin{array}{l}
d_{i,j}\mapsto d_{i,j+1}\\
u_\ell\mapsto u_{\ell-1}\\
u_1\mapsto 0\\
d^*_{i,j}\mapsto d^*_{i,j}
  \end{array}\right. & {\rm for }\; n> 2\tau.

\end{array}
\]
\end{thm}

Notice that we're implicitly choosing an appropriate isomorphism between the group $SFH(-S^3_{K,n})$ and the vector space generated by the $d_{i,j}$'s and the $u_i$'s (see Theorem \ref{Heddenthm}).

There's an interpretation of the maps $\sigma_{\pm}: SFH(-S^3_{K,n})\to SFH(-S^3_{K,n-1})$ in terms of Figure \ref{Heddenfigure}, when $n\ll 0$: fix a chain complex $C$ computing $\HFK(S^3,K)$ and call $(\ltilde{C}_n, \ltilde\de)$ and $(\ltilde{C}_{n-1}, \ltilde\de)$ the two complexes defined in the previous section, computing $SFH(-S^3_{K,n})$ and $SFH(-S^3_{K,n-1})$ starting from $C$. We have two ``obvious'' chain maps $s_\pm: \ltilde{C}_n\to \ltilde{C}_{n-1}$: $s_-$ sends $\x_i\in \ltilde{C}_n$ to $\x_i\in \ltilde{C}_{n-1}$, while $s_+$ sends $\x_i\in \ltilde{C}_n$ to $\x_{i+1}\in \ltilde{C}_{n-1}$. The maps $s_\pm$ induce the two stabilisation maps $\sigma_\pm$ at the homology level.

$s_-$ is the inclusion $\ltilde{C}_{n}\hookrightarrow \ltilde{C}_{n-1}$ that misses the leftmost vertical tile (that is, the copy $C_{1-n}$ of $C$ that's in lowest Alexander degree), while $s_+$ is the inclusion that misses the rightmost vertical tile (the copy $C_1$ of $C$ that lies in highest Alexander degree).

As a corollary (of the proof), we obtain a graded version of the result:

\begin{cor}\label{sigma-grading}
The maps $\sigma_\pm$ are Alexander-homogeneous of degree $\mp \sfrac12$.
\end{cor}

\begin{rmk}\label{torsion_rmk}
Notice that the maps $\sigma_-$ preserve $S_+$ and eventually kill $S_-$, whereas the maps $\sigma_+$ have the opposite behaviour. Moreover, $\sigma_-$ and $\sigma_+$ are injective on the unstable complex for $n\le2\tau$, while they eventually kill it for $n>2\tau$.

Namely, for $n \le 2\tau$, the subcomplex $S_\pm = \bigcup_{m>0}\ker \sigma_\pm^m = \ker \sigma_{\pm}^N$ for some large $N$ (depending on $K$, but not on the slope $n$: any $N>2g(K)$ works), do not depend on the basis we've chosen. For $n<2\tau$, though, the unstable subspace \emph{does} depend on this choice: this reflects the fact that it is a section for the projection map $SFH(-S^3_{K,n})\to SFH(-S^3_{K,n})/(S_++S_-)$.

On the other hand, for $m>2\tau$ the situation is reversed: the unstable complex is the intersection of the kernels of $\sigma_\pm^N$, and $S_\pm$ is a section of the projection map $\ker \sigma_{\mp}^N \to (\ker \sigma_{\mp}^N)/(\ker \sigma_-^N\cap\ker \sigma_+^N)$.
\end{rmk}

The action of $\sigma_\pm$ on the unstable complex is just by degree shift, as in Theorem \ref{stabmaps}.

\section{An apparently new Legendrian invariant}\label{newLeg}

\subsection{Some remarks on $EH(L)$}

Given an oriented Legendrian knot $L$, we define $L^{m,n}$ to be the Legendrian knot obtained from $L$ via $m$ negative and $n$ positive stabilisations.

The main character of the subsection will be an \emph{unoriented} Legendrian knot $L$ in the 3-sphere $S^3$, equipped with some contact structure $\xi$.

\begin{prop}\label{EHntoEH}
If $tb(L)\le 2\tau(K)$, the pair $\{EH(L^{0,n}),EH(L^{n,0})\}$ determines $EH(L)$.
\end{prop}

Strictly speaking, since $L$ is not oriented, $EH(L^{0,n}), EH(L^{n,0})$ are not individually defined, but the pair $\{EH(L^{0,n}),EH(L^{n,0})\}$ is, as the unordered pair $\{\sigma_-^n(EH(L)), \sigma_+^n(EH(L))\}$ for either orientation of $L$.

\begin{proof}
Since $\sigma_-$ preserves $S_-$ and $\sigma_+$ preserves $S_+$, knowing the pair we know what the stable part of $EH(L)$ is.

Let's consider now the unstable component of $EH(L)$: since $EH(L)$ is represented by a single generator in the chain complex, it is Alexander-homogeneous; moreover, since the stable and unstable complexes are generated by homogeneous elements, both the stable and unstable components of $EH(L)$ are Alexander-homogeneous. We now state a proposition that will turn out to be useful later, and we will prove it below.

\begin{prop}\label{stablevsOT}
$\xi$ is overtwisted if and only if $EH(L)$ is stable.
\end{prop}

Now, if $\xi$ is overtwisted, $EH(L)$ is stable, so we're done.

On the other hand, if $\xi=\xi_{\rm st}$, the unstable component of $EH(L)$ is nonvanishing, and -- when fixing either orientation -- has Alexander degree $2\tilde{A}(EH(L)) = \tilde{A}(EH(L^{n,0}))+\tilde{A}(EH(L^{0,n}))$, and this suffices to determine it.
\end{proof}

\begin{rmk}
Proposition \ref{stablevsOT} is a analogue to Theorem 1.2 in \cite{LOSS}, which tells us that $\L^-(L)$ is mapped to $c(\xi)$ by setting $U=1$ in the complex $HFK^-(-S^3,K)$. See also Proposition \ref{EH-prop} below.
\end{rmk}

\begin{proof}[Proof of Proposition \ref{stablevsOT}]
We're first going to prove that if $EH(L)$ is stable, $\xi$ is overtwisted, via the following lemma (which will turn out to be useful also later). Let $\psi_\infty$ denote the gluing map associated to the gluing of the standard neighbourhood of a Legendrian knot (\emph{i.e.} the difference $\T_\infty = Y(1)\setminus {\rm Int}(Y_L)$).

\begin{lemma}\label{stablevspsiinfty}
A homogeneous element $x\in SFH(-S^3_{K,n})$ is stable if and only if $\psi_\infty(x)=0$.
\end{lemma}

\begin{proof}
Consider the Legendrian unknot $L\subset (S^3, \xi_{st})$ with $tb(K_0) = -1$, and stabilise it once (with either sign) to get $L'$. By gluing $\T_\infty$ to either $S^3_L$ or $S^3_{L'}$ we obtain the contact structure $\xi_{st}$ on $S^3$. Observe now that $S^3_{L'}$ is obtained from $S^3_L$ by a stabilisation basic slice: it follows in particular that the union $\T'$ of this basic slice and $\T_\infty$ is a tight solid torus. Honda's classification of tight contact structures of solid tori tells us that $\T'$ is isotopic to $\T_\infty$.

Now the associativity of gluing maps (Theorem \ref{associativityphi}) tells us that, as $\T'$ is isotopic (as a contact manifold) to $\T_\infty$, $\psi_\infty\circ\sigma_\pm = \psi_\infty$.

Suppose that $x$ is stable, then there exists a positive integer $N$ such that $(\sigma_-\circ \sigma_+)^N(x) = 0$, and therefore
\[\psi_\infty(x) = \psi_\infty((\sigma_-\circ \sigma_+)^N(x)) = 0.\]

Suppose now that $x$ is not stable. Then $(\sigma_-\circ\sigma_+)^N(x) \neq 0$ for all $N$. Notice that $\sigma_-\circ\sigma_+$ carries homogenous elements to homogenous elements, and has degree 0. By Theorem \ref{stabmaps}, there is a sufficiently large integer $N$ such that the image of $x$ under $(\sigma_-\circ\sigma_+)^N$ lies in the middle part of the complex. More precisely, it lies in a homogenous component of dimension 1, and in particular $x_N=(\sigma_-\circ\sigma_+)^N(x)$ is the generator of the unstable complex in its Alexander-degree summand.

We claim that $\psi_\infty$ doesn't kill $x_N$.

Now take a knot $L'$ that is Legendrian \emph{with respect to the standard contact structure}, and consider $EH(L')$. From the first part of the proof, we know that, for all $k,\ell\ge0$, $(\sigma_-^k\circ \sigma_+^\ell)(EH(L'))\stackrel{\psi_\infty}{\longmapsto}c(\xi_{\rm st})$. But there are positive integers $k,\ell,m$ such that $(\sigma_-^k\circ\sigma_+^\ell)(EH(L'))$ and $x_{N+m}$ have the same Alexander degree, and are both nonzero. Since they live in the same 1-dimensional summand, they're equal, and in particular $\psi_\infty(x_{N+m}) = \psi_\infty(EH(L')) = c(\xi_{\rm st}) \neq 0$.
\end{proof}

We can now conclude the proof of Proposition \ref{stablevsOT}: recall that Eliashberg \cite{El2} proved that the only tight contact structure on $S^3$ is the standard one, and in particular a contact structure $\xi$ on $S^3$ is overtwisted if and only if $c(\xi)= 0$. By the lemma above, though, $c(\xi) = 0$ if and only if $EH(L)$ is stable.
\end{proof}

We can pin down the Alexander grading of $EH(L)$ using an argument analogous to the one that Ozsv\'ath and Stipsicz use for $\L^-(L)$ \cite{OSt}.

\begin{prop}\label{EHgrading}
Identifying $SFH(-S^3_L) = \HFK(S^3_{tb(L)}(K), \ltilde{K})$ as in Proposition \ref{HFvsSFH}, $EH(L)$ is homogenous of Alexander degree $-r(L)/2$.
\end{prop}

\begin{proof}
In \cite[Theorem 4.1]{OSt}, Ozsv\'ath and Stipsicz compute the Alexander degree of $\L^-(L)$ by a combinatorial argument on an open book compatible with $L$. They obtain that
\[A(\L^-(L)) = \frac12\langle c_1(\s(\x_L))-PD([\mu_L]), [F,\de F]\rangle = \frac{tb(L)-r(L)+1}2,\]
where $\x_L$ is a generator representing $\L^-(L)$ in some Heegaard diagram for $S^3$.

Let's consider the following set up: let $(\Sigma,\ba,\bb,\bc,z,w)$ be a triple Heegaard diagram, where $(\Sigma,\ba,\bb,z,w)$ is obtained from an open book compatible with $L$ as in \cite{LOSS}, so that $\L^-(L)$ is represented by a generator $\x$ in $CFK^-(\Sigma,\bb,\ba,z,w)$. Now define $\bc$ to be obtained from $\bb$ by replacing $\beta_0$ with $L\subset F$ as sitting inside the page of the open book, and positioned with respect to $z,w$ as in Figure \ref{AlexEH}.

Notice that $(\Sigma,\bb,\bc,z,w)$ represents an unknot in $\#^{g-1} (S^1\times S^2)$, therefore we can choose a generator $\bfTheta$ representing the top-dimensional class in $\HFK(\Sigma,\bb,\bc,z,w)$.

Ozsv\'ath and Szab\'o proved in \cite[Section 2]{OSHFK} that, whenever we have a triangular domain $\psi\in\pi_2(\x,\y,\bfTheta)$, then
\begin{equation}\label{OSdegreeshift}
\s(\y)-\s(\x) = (n_w(\psi)-n_z(\psi))PD(\mu).
\end{equation}

\begin{figure}
\begin{center}
\includegraphics{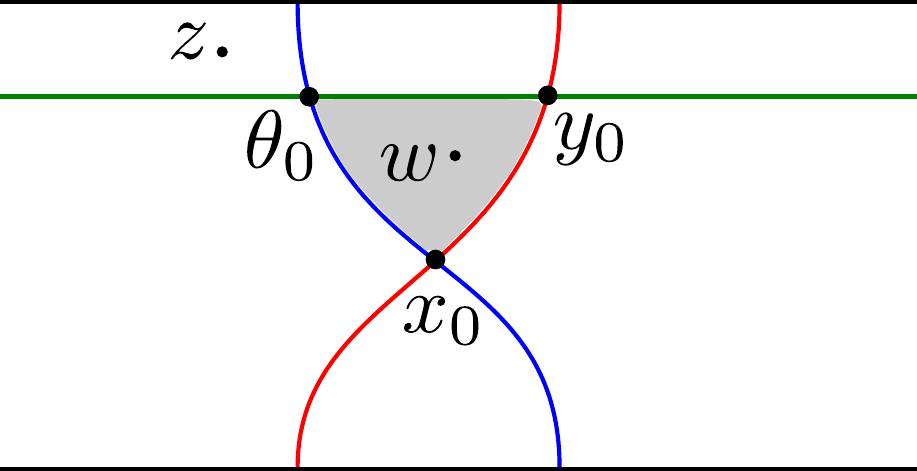}
\caption{The triple Heegaard diagram used in the proof of Proposition \ref{EHgrading}}\label{AlexEH}
\end{center}
\end{figure}

We exhibit in Figure \ref{AlexEH} a Whitney triangle $\psi$ in $\pi_2(\x,\y,\bfTheta)$ with $n_w(\psi) = 1$, $n_z(\psi)=0$ connecting the generator $\x$ in $(\Sigma,\bb,\ba,z,w)$ representing $\L(L)$ and the generator $\y$ for in $(\Sigma,\bc,\ba,D)$ representing $EH(L)$, where $D$ is a disc on $\gamma_0$ that touches the two regions of $\Sigma\setminus(\ba\cup\bc)$ containing $z$ and $w$. Notice that $\x$ and $\y$ live in the \emph{cohomology} groups of $CFL^-(\Sigma,\ba,\bb,z,w)$ and $SFC(\Sigma,\ba,\bc,D) = \CFK(\Sigma,\ba,\bc,z,w)$, so we need to be careful when using Equation \ref{OSdegreeshift}.

More precisely, we want to consider a map $SFC(\Sigma,\bb,\ba)\to SFC(\Sigma,\bc,\ba)$ (we omit basepoint for the sake of clarity), that is dual to a map $SFC(\Sigma,\ba,\bc)\to SFC(\Sigma,\ba,\bb)$ so we should be looking at triangles in the triple Heegaard diagram $(\Sigma,\ba,\bc,\bb)$ instead of $(\Sigma,\ba,\bb,\bc)$. In particular, the grading shifts are reversed: for every triangular domain $D$ in $(\Sigma,\ba,\bb,\bc)$ we associate the domain $-D$ in $(\Sigma,\ba,\bc,\bb)$, so that $n_z$ and $n_w$ change signs.

Since $c_1(\s+\alpha) = c_1(\s)+2\alpha$ for every $\s\in\spin(Y)$ and every $\alpha\in H^2(Y)$, it follows from the computations in \cite[Section 4]{OSt} that
\[\langle c_1(\s(\y)), [F,\de F]\rangle = \langle c_1(\s(\x)), [F,\de F]\rangle-2 = 2A(\L^-(L)) + 1 - 2 = tb(L)-r(L).\]
If we now plug this in Equation \ref{defAlex} and we use $\langle PD([\mu_{\ltilde K}]), [F,\de F]\rangle = tb(L)$, we get
\[A(EH(L)) = \frac{\langle c_1(\s(\y))-PD([\mu_{\ltilde K}]), [F,\de F]\rangle}2 = -\frac{r(L)}2,\]
since $\langle PD([\mu_{\ltilde K}]), [F,\de F]\rangle = tb(L)$ by construction of $S^3_L$.
\end{proof}

We now prove that the hypothesis $tb(L)\le2\tau(K)-1$ above is necessary:

\begin{prop}
For every non-loose unknot $L$ in $S^3$, $EH(L)$ is nonvanishing and purely unstable.
\end{prop}

\begin{proof}
When $K$ is the unknot, the stable complex of $S^3_{K,n}$ is trivial for all values of $n$. Also, $\tau(K)=0$.

According to Eliashberg and Fraser \cite{EF}, $L$ has non-negative Thurston-Bennequin number $tb(L)\ge 0 = 2\tau(K)$, and admits a tight Legendrian surgery $(Y,\xi)$. Since $L$ is topologically unknotted, $Y$ is a lens space, and any tight contact structure on a lens space is Stein fillable: in particular $c(Y,\xi)\neq 0$. Then Lemma \ref{nonvanishing} applies, showing that also $EH(L)\neq 0$.
\end{proof}

We conclude the section by giving an alternative proof of the following fact, due to Etnyre and Van Horn-Morris, and Hedden \cite{EV, He3}. If $K\subset S^3$ is a fibred knot, then it's the binding of an open book for $S^3$, and any fibre is a minimal genus Seifert surface for $K$: call $\xi_K$ the contact structure on $S^3$ supported by this open book.

\begin{thm}
$\xi_K$ is tight if and only if $\tau(K) = g(K)$.
\end{thm}

\begin{proof}
$K$ sits in $\xi_K$ as a transverse knot, and $sl(K)=2g(K)-1$. Let's consider a $\xi_K$-Legendrian approximation $L$ of $K$ such that $tb(L)\ll 0$. Vela--Vick proved that $\Lhat(L)\neq 0$ \cite{VV}, therefore $EH(L)\neq 0$ \cite{SV}. Since $K$ is fibred, $\HFK(S^3,K;g(K))$ is 1-dimensional \cite{OScontact}: using Proposition \ref{EHgrading} above, together with Theorem \ref{Heddenthm} we see that $EH(L)$ is the only nonzero element in the top degree component of $SFH(-S^3_L)$.

If $\tau(K) = g(K)$, then $EH(L)$ is also the generator in top degree of the unstable complex, and in particular $0\neq \psi_\infty(EH(L)) = c(\xi_K)$.

If $\tau(K)<g(K)$, on the other hand, the unstable complex is supported in degree strictly less than $A(EH(L))$, so $0 = \psi_\infty(EH(L)) = c(\xi_K)$.

Thus, applying Eliashberg's classification result \cite{El2} as above, $\xi_K$ is tight if and only if $c(\xi_K)\neq 0$ if and only if $\tau(K)=g(K)$.
\end{proof}

\subsection{The group $\protect\SFH$ }

Let's step back for a second, and consider an oriented topological knot $K$ in $S^3$.

Given a graded vector space $V=\bigoplus_d V_d$, we denote with $V\{s\}$ a graded vector space with graded components $(V\{s\})_d = V_{d-s}$. Consider the family of graded $\F$-vector spaces $\left(A_n := \HFK(-S^3_{-n}(K),\tilde{K})\{(1-n)/2\}\right)$, indexed by integers (notice the $-$ signs in the definition of $A_n$); for each $n$ we have a degree 0 map $\sigma_-: A_n\to A_{n+1}$, the (negative) stabilisation map, induced by the negative basic slice attachment; these data can be conveniently summarized in a direct system $\mathbf{A}_- := ((A_n), (\psi_{m,n})_{n\ge m})$, where the map $\psi_{m,n}:A_m\to A_n$ is $\sigma_-^{n-m}$.

\begin{defn}
Let $\SFH(-S^3,K)$ to be the direct limit $\varinjlim \A_\sigma$, and call $\iota_n$ the universal map $\iota_n: A_n\to \SFH(-S^3,K)$.
\end{defn}

\begin{rmk}
Since we're taking a direct limit, what counts is just what happens for sufficiently large indices. In particular, we just need to know what happens for $n\ge n_0 := -2\tau(K)+1$: this also fits in the picture of contact topology, since this is the only interval where $EH(L)$ can live for a Legendrian $L$ in $(S^3,\xi_{\rm st})$.

What happens for other indices is that, with respect to the maps $\psi_{m,n}$, the only component that survives is $S$: this is going to be more precise below, even though we discuss just the interval $n\ge n_0$.
\end{rmk}

As defined, $\SFH$ is just a graded $\F$-vector space: using the other (\emph{i.e.} the positive) stabilisation map $\sigma_+$, we can endow it with an $\F[U]$-module structure. One way to do it is to identify the projective limit with the quotient of the disjoint union $\coprod A_n$ by the relations $x_i \sim x_j$ whenever there exists $N$ such that $\psi_{i,N}(x_i) = \psi_{j,N}(x_j)$ and defining $U\cdot [x] = [\sigma_+(x)]$: since $\sigma_-$ and $\sigma_+$ commute, the map is well defined. Notice that the map $\sigma_+$ has now Alexander degree $-1$ (due to the degree shift introduced), and so does the map $U\cdot$ on $HFK^-(S^3,K)$.

Alternatively, we can see the map induced by $\sigma_+$ in a more abstract (and universal) way, considering the following diagram:

\[\xymatrix{& \SFH(-S^3,K) &\\
& \SFH(-S^3,K)\ar@{-->}[u]^{U\cdot} & \\
A_m\ar@/^1pc/[uur]^{\iota_m\circ\sigma_+}\ar[rr]_{\psi_{m,n}}\ar[ur]^{\iota_m} & & A_n\ar@/_1pc/[uul]_{\iota_n\circ\sigma_+}\ar[ul]_{\iota_n}
}.\]

Ignoring the dashed arrow, the diagram commutes, since $\sigma_-$ and $\sigma_+$ commute, and by the universal property of the direct limit (and of the arrows $\iota_n$!), there's a uniquely defined map $U\cdot$, that is the dashed arrow.

\begin{rmk}
We have a dual direct system $\A_+$ defined using $\sigma_+$ rather than $\sigma_-$, and changing the sign of the degree shift.

Reversing the orientation of $K$ induces, as expected, an isomorphism of $\F$-vector spaces $\SFH(-S^3,K) \simeq \SFH(-S^3,-K)$: this follows from the fact that the two direct systems $\A_-$ and $\A_+$ are isomorphic. Moreover, the universal isomorphism commutes with the $U$-action, and this $U$-equivariance gives the isomorphism in the category of $\F[U]$-modules.

This symmetry can also be seen as a choice for the labelling of positive \emph{vs} negative stabilisation, which is in fact equivalent to the choice of an orientation.
\end{rmk}

\begin{thm}\label{isoSFH}
The groups $\SFH(-S^3,K)$ and $HFK^-(-S^3,K)$ are isomorphic as $\F[U]$-modules.
\end{thm}

Before diving into the proof, recall Ozsv\'ath and Szab\'o's description of $HFK^-$ (see for example \cite{OSinteger}, especially Figures 1 and 2). The complex is a direct sum of countably many copies of $\HFK$, each thought of as $U^k\cdot\HFK$ for $k\in\N$: this gives the complex the $\F[U]$-structure; we think of each copy drawn as a vertical tile of Alexander-homogeneous components, and that all copies stacked in the plane like a staircase parallel to the $x=y$ diagonal; the differential comes from the complex $(\HFK,\de)$ computing $\HF(S^3)$, and it can be depicted as a set of arrows pointing horizontally, each coming from a vertical arrow in $(\HFK,\de)$ and corresponding to a domain crossing the auxiliary basepoint $w$. There's a quite striking similarity between the first chunks of this complex and the first chunks of the complexes computing $A_n$'s, and this similarity is both the inspiration and the key of the proof of the theorem.

\begin{proof}
We'll split the proof in two steps: first we'll prove the isomorphism of the two as graded $\F$-vector spaces, and then as $\F[U]$-modules. As usual, we'll call $g=g(K)$ and $\tau=\tau(K)$.

\vskip 0,4 cm

\textbf{Step 1}. We want to prove there are maps $j_n: A_n\to H := HFK^-(-S^3,K)$ such that $(H,\{A_n\},\{j_n\})$ satisfy the universal property for the direct limit of $\A_-$:
\[\xymatrix{& C &\\
& H\ar@{-->}[u]^{\phi} & \\
A_m\ar@/^1pc/[uur]^{\phi_m}\ar[rr]_{\psi_{m,n}}\ar[ur]^{j_m} & & A_n\ar@/_1pc/[uul]_{\phi_n}\ar[ul]_{j_n}
}.\]

We need to define the maps $j_n$ first, and then we need to prove that for every commutative diagram with maps $\phi_n$ to a module $C$ there is a unique (dashed) map $\phi$ making the full diagram commute.

The maps $j_n$ are easily defined: thanks to the previous description, $HFK^-(-S^3,K)$ is the direct sum of a copy of $S_-\subset A_n$ and a copy of $\F[U]$, with $A(U^k) = \tau-k$; imagining a superposition between the two pictures for the complexes computing $A_n$ and $H$ yields to the claim that $j_n$ would like to be a fixed (\emph{i.e.} not depending on $n$) graded isomorphism on $S_-$, zero on $S_+$ and the degree 0, injective map $U_n\to \F[U]$: the commutativity of the lower triangle of the diagram is clear by the description of the maps $\sigma_\pm$.

Now we can consider the full diagram, and show that $\phi$ is uniquely defined by $(\phi_n)_{n\ge n_0}$: consider an element $x_m = a_m+s_m\in A_m$, with $s_m\in S_+$ and $a_m\in S_-\oplus U_m$, and consider the diagram for $n=m+d\gg m$: since the lower triangle is commutative, we have that
\[\phi_m(x_m) = \phi_n(\sigma_-^d(x_m)) = \phi_n(\sigma_-^d(a_m)) = \phi_m(a_m),\]
so $\phi_m(S_+) = 0$: this implies that the map $\phi_m$ factors through $j_m$.

Now, define $\phi$ by $\phi|_{S_-} = \phi_m|_{S_-}$ for some $m$ and $\phi|_{\F[U]/(U^m)} = \phi_m\circ j_m^{-1}$: notice how $\phi$ is well defined (since $\sigma_-$ is an isomorphism on $S_-$ and the injection of degree $+\sfrac12$ on the unstable complex), and makes the diagram commute.

Since $j_m$ is injective on $S_-\oplus U_m$ and $\F[U]$ is the direct limit of $\F[U]/(U^k)$, this is the only way we can define $\phi$, and this concludes the first part of the proof.

\begin{rmk}\label{iotan}
It's worth remarking explicitly what we've proven: we've shown that the inclusion map $\iota_n: \HFK(-S^3_{-n}(K),\tilde{K})\to \SFH(-S^3,K)$ is injective on $S_-\oplus U$, and that $S_+ = \ker\iota_n$ for each $n\ge n_0$. Moreover, for $n$ sufficiently large, the map $\iota_n$ is an isomorphism between truncations of $A_n$ and $HFK^-(-S^3,K)$ that forgets of all elements of low Alexander degree.
\end{rmk}

\vskip 0,4 cm

\textbf{Step 2}. We now need to prove that the two $\F[U]$-module structure correspond under some map: we just need to show that the universal map $\Phi$ in the diagram
\[\xymatrix{& H &\\
& \SFH(-S^3,K)\ar[u]^{\Phi} & \\
A_m\ar@/^1pc/[uur]^{j_m}\ar[rr]_{\psi_{m,n}}\ar[ur]^{\iota_m} & & A_n\ar@/_1pc/[uul]_{j_n}\ar[ul]_{\iota_n}
}\]
is $U$-equivariant, since the universal property for $(H,\{A_n\},\{\iota_n\})$ already implies that it's an $\F$-isomorphism. For $x\in A_n$, the map $\Phi$ sends $\iota_n(x)$ to the class $[x]=j_n(x)$.

We have a good way to picture $\Phi$ when the framing is large: in this case, we just superpose the picture of the complex described in Section \ref{modules} above with Ozv\'ath and Szab\'o's description, and identify generators pointwise. But we're working with the projective limit $\SFH$, which is not the disjoint union $\coprod A_n$, but rather its quotient by the relation $x\sim \psi_{m,n}(x)$. Up to changing the choice of $n$ and $x$, we can suppose that Theorem \ref{Heddenthm} above applies: in this case, the map $\sigma_+$ is just an injection of $A_n$ \emph{on the bottom} of $A_{n+1}$, which, in Ozsv\'ath and Szab\'o's picture corresponds to shifting each copy $U^k\cdot\HFK(-S^3,K)$ to the next one, $U^{k+1}\cdot\HFK(-S^3,K)$, hence proving the $U$-equivariance of $\Phi$.
\end{proof}

\subsection{$\protect\EH$ invariants}

Suppose now we have an oriented Legendrian knot $L$ in $(S^3,\xi)$, of topological type $K$: by construction, we have a naturally defined \emph{oriented} contact class in $\SFH(-S^3,K)$.

\begin{defn}
Define the class $\EH(L)\in \SFH(-S^3,K)$ as $[EH(L)]$, in the identification $\SFH(-S^3,K) = \coprod A_n/\sim$.
\end{defn}

We can immediately read off some facts about this new invariant, that follow straight away from the definition:

\begin{prop}\label{EH-prop}
Consider an oriented Legendrian $L$ in $(S^3,\xi)$ of topological type $K$; then:
\begin{itemize}
 \item[(i)] for a \emph{negative} stabilisation $L'$ of $L$, $\EH(L') = \EH(L)$;

 \item[(ii)] for a \emph{positive} stabilisation $L''$ of $L$, $\EH(L') = U\cdot \EH(L)$;

 \item[(iii)] $\EH(L)$ is an element of $U$-torsion if and only if $\xi$ is overtwisted.
 
 \item[(iv)] $\EH(L)$ sits in Alexander grading $\displaystyle\frac{tb(L)-r(l)+1}2$.
\end{itemize}
\end{prop}

\begin{proof}
\begin{itemize}
\item[\emph{(i)}] $L'$ is a negative stabilisation of $L$, so $EH(L') = \sigma_-(EH(L))$, and
\[\EH(L') = [EH(L')] = [\sigma_-(EH(L))] = [EH(L)] = \EH(L).\]

\item[\emph{(ii)}] $L''$ is a positive stabilisation of $-L$, so $EH(L'') = \sigma_+(EH(L))$, and
\[\EH(L'') = [EH(L'')] = [\sigma_+(EH(L))] = U\cdot[EH(L)]= U\cdot \EH(L).\]

\item[\emph{(iii)}] By definition, an element $[x]$ of $\SFH(-S^3,K)$ vanishes if and only if $\sigma_-^k(x) = 0$ for some $k$, and is of $U$-torsion if and only if $[\sigma_+^h(x)] = 0$ for some $h$: in particular, since $\sigma_-$ and $\sigma_+$ commute, $[x]$ is of $U$-torsion if and only if $(\sigma_-\circ\sigma_+)^\ell(x) = 0$ for some $\ell$. If $tb(L)>2\tau(K)$ (and therefore $\xi$ is overtwisted), we know that $SFH(-S^3_L) = \ker (\sigma_-\circ\sigma_+)^\ell$, so in particular $EH(L)$ is $U$-torsion. On the other hand, if $tb(L)<2\tau(L)$, Lemma \ref{stablevsOT} tells us that $(\sigma_-\circ\sigma_+)^\ell(EH(L))$ vanishes if and only if $\xi$ is overtwisted.

\item[\emph{(iv)}] $EH(L)$ lives in the group $\HFK(-S^3_{tb(L)}(K), \ltilde{K})$, and by Proposition \ref{EHgrading}, its Alexander degree is $-r(L)/2$. Therefore, it lives in degree $\frac{tb(L)-r(L)+1}2$ in $A_{-tb(L)}$ and in $\SFH(-S^3,K)$.
\end{itemize}
\end{proof}

\begin{rmk}\label{orientation_rmk2}
$EH(L)$ is an \emph{unoriented} invariant, \emph{i.e.} doesn't see orientation reversal, whereas the sign of the stabilisation does (see Remark \ref{orientation_rmk}), so one apparently can find a contradiction in Proposition \ref{EH-prop}. What happens is that when we reverse the orientation of $L$, we also reverse the orientation of $K$ and we swap the r\^oles the two maps $\sigma_-$ and $\sigma_+$ play. The two resulting groups, associated to $\A_-$ and $\A_+$ are -- as already noticed -- isomorphic, but in the first one $\sigma_-$ acts trivially and $\sigma_+$ acts as $U$ (as seen in the proof of Proposition \ref{EH-prop}.\emph{(i,ii)}), while in the second one we'd have to write:
\[\begin{array}{l}\EH(L') = [EH(L')] = [\sigma_-(EH(L))] = U\cdot[EH(L)],\\
\EH(L'') = [EH(L'')] = [\sigma_+(EH(L)] = [EH(L))].\end{array}\]
\end{rmk}

\subsection{Transverse invariants}\label{trans_inv}

Let's just recall the classical theorem relating transverse and Legendrian knots: it will be the key fact throughout this subsection.

\begin{thm}[\cite{EH}]\label{ttL}
Two transverse knots are transverse isotopic if and only if any two of their Legendrian approximations are Legendrian isotopic up to negative stabilisations.
\end{thm}

As it happens for $\L^-$, also $\EH$ descends to a transverse isotopy invariant of transverse knots:

\begin{defn}
Given a transverse knot $T$ in $(S^3,\xi)$ of topological type $K$, we can define $\EH(T) = \EH(L)$ for a Legendrian approximation $L$ of $T$.
\end{defn}

The transverse element is well-defined, in light of Proposition \ref{EH-prop} and Theorem \ref{ttL}. A stronger statement holds, the natural counterpart of Proposition \ref{EHntoEH}, that reveals a transverse nature of $EH$:

\begin{thm}
Suppose $L, L'$ are two oriented Legendrian knots in $S^3$ that have the same classical invariants. Suppose also that both the transverse pushoffs of $L, L'$ and the ones of $-L,-L'$ are transversely isotopic. Then $EH(L) = EH(L')$.
\end{thm}

\begin{proof}
Since the pushoffs of $L$ and $L'$ (respectively, of $-L$ and $-L'$) are transverse isotopic, $\EH(L)=\EH(L')$ (resp. $\EH(-L)=\EH(-L')$). By Remark \ref{iotan}, and by the behaviour of $\sigma_\pm$ on the unstable complex, we can reconstruct all three components (that is, along $S_\pm$ and $U$) of $EH(L)$ from $\EH(L)$ and $\EH(-L)$, and this concludes the proof.
\end{proof}

\section{$\protect\EH$ \emph{vs} $\L^-$}\label{last}

Fix an oriented Legendrian knot $L$ in $(S^3,\xi)$, of topological type $K$: the LOSS invariant $\L^-(L)$ is an element of $HFK^-(-S^3,K)$, which has just been proven isomorphic to $\SFH(-S^3,K)$, where $\EH(L)$ lives. Let's also recall the following theorem:

\begin{thm}\cite[Theorems 1.2 and 1.6]{LOSS}
For $L$ as before:
\begin{itemize}
 \item[(i)] for a \emph{negative} stabilisation $L'$ of $L$, $\L^-(L') = \L^-(L)$;

 \item[(ii)] for a \emph{positive} stabilisation $L''$ of $L$, $\L^-(L'') = U\cdot \L^-(L)$;

 \item[(iii)] $\L^-(L)$ is an element of $U$-torsion if and only if $\xi$ is overtwisted.
 
 \item[(iv)] $\L^-(L)$ sits in Alexander degree $\frac{tb(L)-r(L)+1}2$.

\end{itemize}
\end{thm}

Notice how the theorem above is formally identical to our Proposition \ref{EH-prop}: it's therefore natural to compare the two invariants $\EH$ and $\L^-$.

\begin{thm}\label{EH=L-}
Given $L$ as before, there's an isomorphism of bigraded $\F[U]$-modules $\SFH(-S^3,K) \to HFK^-(-S^3,K)$ taking $\EH(L)$ to $\L^-(L)$.
\end{thm}

We postpone the proof of the main theorem to the last subsection, and draw some conclusions from the theorem, first.

It's now worth stressing and making precise what we've announced in the introduction, that $EH(L)$ (but \emph{not} $\EH(L)$!) contains at least as much information as $\L^-(L)$ and $\L^-(-L)$ together. We can prove the following refinement of Theorem \ref{mainthm}:

\begin{thm}
For two oriented Legendrian knots $L_0,L_1$ in $(S^3,\xi)$ of topological type $K$, with $tb(L_0), tb(L_1)\le 2\tau(K)$, the following are equivalent:
\begin{itemize}
\item[(i)] $EH(L_0) = EH(L_1)$;
\item[(ii)] $\L^-(L_0) = \L^-(L_1)$ and $\L^-(-L_0) = \L^-(-L_1)$.
\end{itemize}
In general, withouth any restriction on the Thurston-Bennequin numbers of $L_0$ and $L_1$, (i) implies (ii).
\end{thm}

\begin{proof}
$EH(L_i)$ determines both $\L^-(L_i)$ and $\L^-(L_i)$ by Theorem \ref{EH=L-}, so \emph{(ii)} follows from \emph{(i)}.

Let's now suppose that the constraint on the Thurston-Bennequin invariants holds. As already observed (see Remarks \ref{orientation_rmk} and \ref{orientation_rmk2}), $\EH$ is an oriented invariant of Legendrian knots with the following property: $\EH(L_1) = \EH(L_2)$ if and only if the components of $EH(L_1)$ and $EH(L_2)$ along $S_-$ and $U$ agree. In particular, if $\L^-(L_0) = \L^-(L_1)$ the components of $EH(L_i)$ along $S_-$ and along $U$ are equal; if $\L^-(-L_0) = \L^-(-L_1)$, then the $S_+$ components of $EH(L_1)$ and $EH(L_2)$ agree, too, thus showing that $EH(L_1) = EH(L_2)$.
\end{proof}

\subsection{Triangle counts}

The proof of Theorem \ref{EH=L-} relies on bypass attachments on contact sutured knot complements and the induced gluing maps, henceforth called simply \deff{bypass maps}.

There is another description of a sutured manifold with torus boundary and annular $R_+$ we're going to need: an \deff{arc diagram} $\H^a$ is a quintuple $(\Sigma,\ba,{\beta^a},\bb^c,D)$, where $\Sigma$ is a closed surface, $\ba$ and $\bb^c$ are sets of non-disconnecting, simple closed curves in $\Sigma$, $D$ is a closed disc disjoint from $\ba\cup\bb^c$ and $\beta^a$ is an arc properly embedded $\Sigma\setminus ({\rm Int}(D)\cup\bb^c)$. We further ask that $|\ba| = g = g(\Sigma)$, and $|\bb^c|= g-1$. We will often drop $D$ from the notation and write $\bb$ for $\bb^c\cup\{\beta^a\}$, for sake of brevity.

We build a sutured manifold $(M,\Gamma)$ with torus boundary and two parallel sutures out of $\H^a$ as follows: the set of $\alpha$-curves determines how to attach $g$ upside-down 2-handles on $\Sigma\times\{0\}\subset\Sigma\times[0,1]$; we attach a 0-handle (a ball) to fill up the remaining component of the lower boundary; the set $\bb^c$ of $\beta$-\emph{curves} determines the attaching circles of 2-handles on $\Sigma\times\{1\}$. We define $M$ to be the manifold obtained by smoothing corners after these handle attachments; notice that $D$ is an embedded disc in $\de M$, and $\beta^a$ is an embedded arc in $\de M$. Let $R_+$ be a small regular neighbourhood of $D\cup\beta^a$ and $\Gamma$ be its boundary.

We can now consider the chain complex $SFC(\H^a)$ as usual, by taking $g$-tuples of intersection points of $\alpha$-curves and $\beta$-curves and arcs, so that no two points lie on the same curve or arc, and the differential counts holomorphic discs whose associated domains do not touch the disc $D$. It is clear that $SFC(\H^a)$ is isomorphic as a chain complex to the complex associated to a doubly-pointed Heegaard diagram representing the dual knot $\ltilde K$ inside $S^3_{\gamma}(K)$; in particular, it is also chain homotopic to a complex computing $SFH(M,\Gamma)$.

\begin{rmk}
The construction above is related to Zarev's bordered sutured manifolds and their bordered sutured diagrams \cite{Zarev}, and in fact generalises to sutured manifolds with connected $R_+$. What we called arc diagrams are in fact similar to bordered sutured diagrams (but not to what he calls arc diagrams).
\end{rmk}

In order to obtain the bypass maps we need to count holomorphic triangles in triple arc diagrams. At the level of arc diagrams, attaching a bypass to $(M,\Gamma)$ corresponds to choosing another arc $\gamma^a$ on $\Sigma$, which intersects $\beta^a$ transversely in a single point $\theta^a$. Every $\gamma$-curve is a small perturbation of a $\beta$-curve in $\H_\beta$, and therefore there is a preferred choice among the two intersection points (see \cite{OSHF} and Section \ref{beta} below), giving an element $\bfTheta$. We then have:

\begin{thm}[\cite{JakeHKM}]\label{Jakebypassprop}
The bypass map is induced by the triangle count map $F(\cdot\otimes\bfTheta)$ associated to the triple diagram described above.
\end{thm}

Somewhat confusingly, the r\^oles of $\alpha$- and $\beta$-curves are reversed when talking about contact invariants, since we're looking at elements in $SFH(-M,-\Gamma)$ rather than in $SFH(M,\Gamma)$: we will be very explicit and careful about the issue of triangle counts in this setting, as we discuss below.

Recall that, given a Legendrian knot $L$ in any contact 3-manifold, $EH(L)$ is the class of a generator $\x_{EH} \in SFC(\Sigma, \bb, \ba)$, where $\H = (\Sigma,\ba,\bb)$ is an arc diagram representing $S^3_L$ (notice the order of $\ba$ and $\bb$).

\subsubsection{$\alpha$-slides}\label{alpha} If we want to do a handle-slide among the $\alpha$-curves in $\H$, say changing $\ba =\{\alpha_i\}$ to $\ba' = \{\alpha_i'\}$, what we are doing is replacing the \emph{second} set of curves in a (doubly-pointed) Heegaard diagram. A triangle count in $(\Sigma,\bb,\ba,\ba')$ gives a map
\[
CF(\bb,\ba)\otimes CF(\ba,\ba') \otimes CF(\ba',\bb) \to \F,
\]
which in turn gives a map
\[
F_{\alpha\alpha'}: CF(\bb,\ba)\otimes  CF(\ba,\ba') \to CF(\bb,\ba').
\]
(Here we've been dropping $\Sigma$ from the notation, and we'll do it again later.)

In all cases we're going to meet, the top-dimensional generator in $HF(\ba,\ba')$ will be represented by a single generator, that we call $\bfTheta_{\alpha\alpha'}$, and the map we'll be looking at is $\Psi_{\alpha\alpha'}: F_{\alpha\alpha'}(\cdot\otimes\bfTheta_{\alpha\alpha'})$.

Consider a holomorphic triangle $\psi$ connecting $\x$ to $\y$ giving a nontrivial contribution to the $\Psi_{\alpha\alpha'}$; that is, the Maslov index of $\psi$ is 0 and the moduli space contains an odd number of points. The boundary of the domain $\D(\psi)$ associated to $\psi$ has the following behaviour along its boundary:
\begin{itemize}
\item[A1.] $\de\de_\alpha \D(\psi) = \x-\bfTheta_{\alpha\alpha'}$;
\item[A2.] $\de\de_{\alpha'} \D(\psi) = \bfTheta_{\alpha\alpha'}-\y$;
\item[A3.] $\de\de_\beta \D(\psi) = \y-\x$.
\end{itemize}
This amounts to saying that if we travel along $\de\D(\psi)$ following the orientation induced by $\D(\psi)$ we cyclicly run along curves in the order $\beta, \alpha', \alpha$.

\subsubsection{$\beta$-slides}\label{beta} On the contrary, if we're doing some triangle count that changes the $\beta$-curves or arcs instead (as we will see below), we are going to face the opposite behaviour. More precisely, consider a set of curves $\bb'$. A triangle count in $(\Sigma,\bb,\ba,\bb')$ gives a map
\[
CF(\bb,\ba)\otimes CF(\ba,\bb') \otimes CF(\bb',\bb) \to \F,
\]
which in turn gives a map
\[
F_{\beta\alpha\beta'}: CF(\bb,\ba)\otimes  CF(\bb',\bb) \to CF(\bb',\ba).
\]
In all cases we're going to meet, the top-dimensional generator in $HF(\bb',\bb)$ will be represented by a single generator, that we call $\bfTheta_{\beta'\beta}$, and the map we'll be looking at is $\Psi_{\beta\beta'}: F_{\alpha\alpha'}(\cdot\otimes\bfTheta_{\beta\beta'})$. Notice that $\bfTheta_{\beta'\beta}$ represents the bottom-dimensional generator of $HF(\bb,\bb')$.

Therefore, if we call $\psi$ a triangle as above, giving a nontrivial summand $\y$ in $F_{\beta\beta'}(\x)$ connecting generators $\x$ and $\y$, we get the following conditions on $\D(\psi)$:
\begin{itemize}
\item[B1.] $\de\de_\alpha \D(\psi) = \x-\y$;
\item[B2.] $\de\de_{\beta} \D(\psi) = \bfTheta_{\beta'\beta}-\x$;
\item[B3.] $\de\de_{\beta'} \D(\psi) = \y-\bfTheta_{\beta'\beta}$.
\end{itemize}
That is to say that moving along $\de\D(\psi)$ following the orientation induced by $\D(\psi)$ we meet the curves $\beta, \beta',\alpha$.

\subsection{Proof of Theorem \ref{EH=L-}}

The idea underlying the proof is to find explicit representatives for the two contact invariants $\EH(L)$ and $\L^-(L)$ that live in suitable Heegaard diagram, and compare them.

The proof will be divided in three steps:

\begin{enumerate}
\item We construct an open book $(S^3,\xi,L')$ for a single negative stabilisation $L'$ of $L$, together with an associated arc diagram $\H^{sut}$ and an associated doubly-pointed Heegaard diagram $\H^{knot}$, representing $SFH(-S^3_{L'})$ and $HFK^-(-S^3, L')$ respectively.
\item We consider a large negative stabilisation $L^{stab}$ of $L'$. Stabilisations corresponds to bypass attachments on $\H^{sut}$: we compute the associated triangle counts, obtaining a generator in a diagram $\H^{stab}$, representing $EH(L^{stab})$. Moreover, $\H^{stab}$ has a handle that is very similar to the \deff{winding region} (see Figure \ref{Heddentype}).
\item Finally, we handle-slide a single $\alpha$-curve and compare $\H^{stab}$ with $\H^{knot}$ using a refinement of a result of Hedden \cite{He}.
\end{enumerate}

\begin{figure}[ht]
\begin{center}
\includegraphics{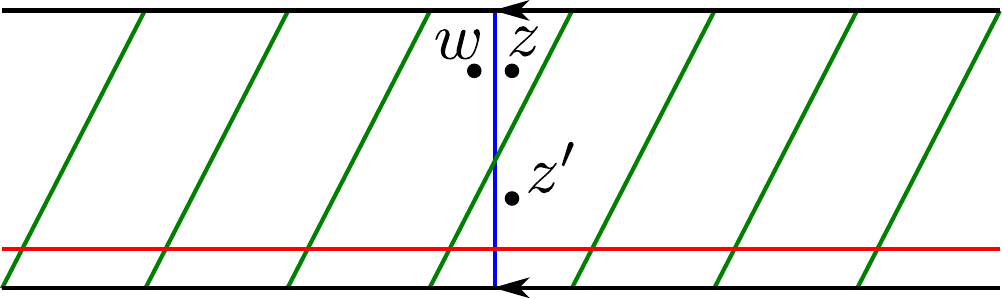}
\end{center}
\caption{The winding region: the picture represent a handle, with the top and the bottom sides of the rectangle identified according to arrows. The horizontal curve (in red) is a $\alpha$-curve, the vertical curve (in blue) is a $\beta$-curve, representing the meridian for the knot in $S^3$, whereas the curve that winds along the handle (in green) is the $\gamma$-curve representing the given framing on the boundary of $S^3\setminus\nu(K)$: basepoints are placed such that $(\Sigma,\ba,\bb,z,w)$ represents $(S^3,K)$, while $(\Sigma,\ba,\bc,z',w)$ represents $(S^3_n(K),\tilde{K})$.}\label{Heddentype}
\end{figure}

\begin{proof}
\textbf{Step 1.} Recall the definition of $\L^-$: given $L\subset (S^3,\xi)$, using an idea of Giroux (\cite{Gi}, see also \cite{Etn}) we can construct an open book $(F,\phi)$ with $L$ sitting on one of the pages (identified with $F$, so that $L\subset F$) as a homologically nontrivial curve. We then choose a basis for $F$ (in the sense of Definition \ref{basis}) with only one arc, say $a_1$, intersecting $L$. We can construct a doubly-pointed Heegaard diagram as we did for the $EH$-diagram, the only thing to take care of being placing the two basepoints (see \cite{LOSS}). A representative for $\L^-$ is now given by the only intersection point entirely supported in $F\subset\Sigma$.

The following lemma is implicitly used by Stipsicz and V\'ertesi \cite{SV}.

\begin{lemma}
The partial open book $(S,P,h) := (F,F\setminus \nu(L),\phi|_P)$ represents the manifold $(S^3_L,\xi|_{S^3_L})$, where $\nu(L)$ is a small neighbourhood of $L$ in $S^3$.
\end{lemma}

\begin{proof}
The contact sutured manifold $(M,\Gamma)$ associated to $(S,P,h)$ embeds in $S^3$ as the complement of a small neighbourhood of $L$, since we can embed the two halves of $M$ inside the two halves of $S^3$ given by $(F,\phi)$, respecting the foliation: this shows that $(M,\Gamma)$ is contactomorphic to $S^3_L$.
\end{proof}

\begin{rmk}
We can read off a sutured Heegaard diagram associated to $(S,P,h)$ directly from the doubly-pointed Heegaard diagram for $(S^3,L)$: we just need to remove the basepoints, together with a (small, open) neighbourhood of $L$ in the Heegaard surface, and erase the two curves corresponding to $a_1$ and $b_1$. The remaining $a_i$'s form a basis for the $(S,P,h)$, so the $EH$ invariant is already on the picture.

If we also want to have an arc diagram for $S^3_L$, we can to do the following. We start with the doubly-pointed Heegaard diagram, and replace the curve $\beta_1$ with a curve $\lambda$ parallel to $L$. Then we add a disc $D$ along it this curve, that is disjoint from $\alpha_1$ and lies in the two regions that are occupied by the basepoints. Finally, we just forget about the basepoints and let $\beta^c_1$ be the arc with endpoints in $D$ that runs along $\lambda$. Notice that this arc arc intersects a single $\alpha$-curve (namely, $\alpha_1$) exactly once.

Notice that in this case the chain complexes associated to the sutured Heegaard diagram and the arc diagram are trivially isomorphic (as chain complexes), since they have exactly the same generators and count precisely the same curves (since the arc intersects only $\alpha_1$, and in a single point).
\end{rmk}

\begin{figure}[!ht]
\begin{center}
\includegraphics[scale=0.75]{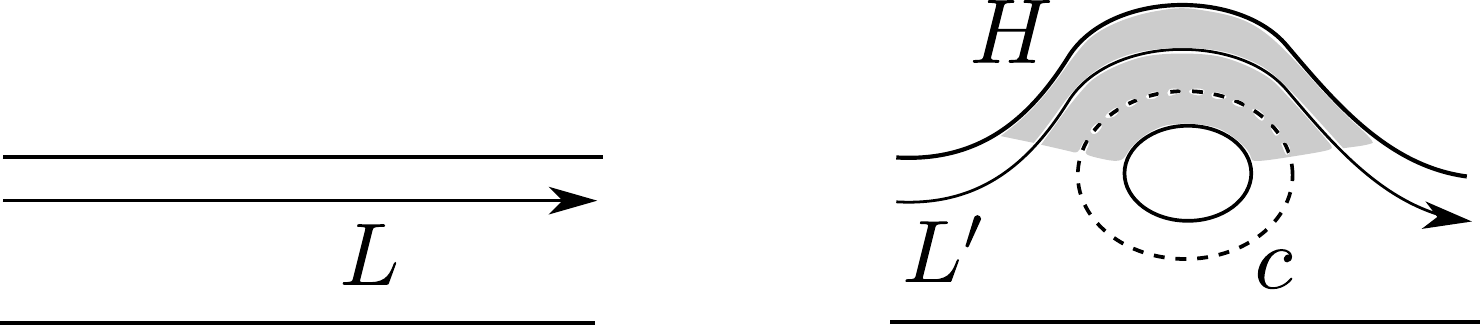}
\end{center}
\caption{On the left we have a 1-handle of the page of an open book for $(Y,\xi,L)$, where the arrow represents $L$. On the right we have the page with the additional handle $H$ (shaded), the curve $c$ along which we perform a positive Dehn twist; the arrow represents $L'$ (which otherwise agrees with $L$).}\label{stab_fig}
\end{figure}

We now want to know what happens to this picture when we stabilise $L$ negatively to get $L'$. If $L$ sits on a page $F$ of the open book $(F,\phi)$, $L'$ sits on a page of the open book $(F',\phi') = (F\cup H, \phi\circ \delta_c)$, where $H$ is a 1-handle attached to the boundary of $F$ as in Figure \ref{stab_fig} and $\delta$ is a positive Dehn twist along the curve $c$, dashed in the figure. $L'$ is isotopic to $L$ inside $F$, except that it runs once along the handle \cite{On}.

Let's see what happens at the level of arc diagrams: recall that the invariant $EH(L')$ is represented by a chain $\x_{EH}$ in the arc diagram $\H^a = (\Sigma,\beta^a,\bb^c,\ba)$ coming from the open book $(F',\phi')$ together with the embedding $L'\subset F'$. In particular we have that $\Sigma = F'\cup -F'$ and $D\cup\beta^a \subset F' \subset \Sigma$. Call $g+1$ the genus of $\Sigma$; the $\alpha$-curves are obtained after choosing a basis $\{a_0,a_1,\dots,a_g\}$ for $F$. We choose this basis so that $a_0$ is the co-core of the handle $H\subset F'$, and is the only arc intersecting $L'$ inside the page, and $a_1$ is the only other arc intersecting the curve $c$ above (this is always possible). Finally, we let $\beta_0 = \beta^a$ the arc that runs parallel to $L'$ inside $F'$, $\alpha_0 = a_0\cup -a_0$, which is the only curve that intersects $\beta_0$, $\alpha_1 = a_1\cup -a_1$, and we number the remaining curves so that $\alpha_i$ and $\beta_i$ intersect once inside $F'$. Recall that $\x_{EH}$ is the generator consisting of all the intersection points $x_i = \alpha_i\cap\beta_i$ inside $F'$.

\vskip 0,3 cm

\textbf{Step 2.} We now want to attach bypasses to the sutured knot complement $S^3_{L'}$ and compute the associated gluing maps, as indicated in Theorem \ref{Jakebypassprop}.

When we stabilise $L'$ we attach a bypass to the sutured knot complement, and the framing of the sutures decreases by 1. We're going to obtain an arc diagram $(\Sigma,\gamma^a,\bc^c,\ba)$ for a stabilisation $L''$ of $L'$ by attaching a bypass to the sutured knot complement $S^3_{L'}$ (see \cite{SV}): the Heegaard surface $\Sigma$ and the curves $\ba$ are the same as in $\H^a$; also, $\bc^c = \bb^c$. The curve $\gamma_0$ is obtained by juxtaposing $\beta_0$ and $\mu$ as in Figure \ref{stab_3HD}, where $\mu$ is the meridian of $L\subset S^3$. Notice that $\mu$ can be obtained by taking $a_0 = \alpha_0\cap F'$ and letting $\mu = a_0 \cup -\phi'(a_0)$; in other words, $\mu$ is the curve $\beta_0$ in the doubly-pointed Heegaard diagram of representing $HFK^-(-S^3,L)$. Observe also that the arc $\gamma_0$ intersects $\beta_0$ transversely in a single point, $\theta_0$.

\begin{figure}
\begin{center}
\includegraphics[scale=0.8]{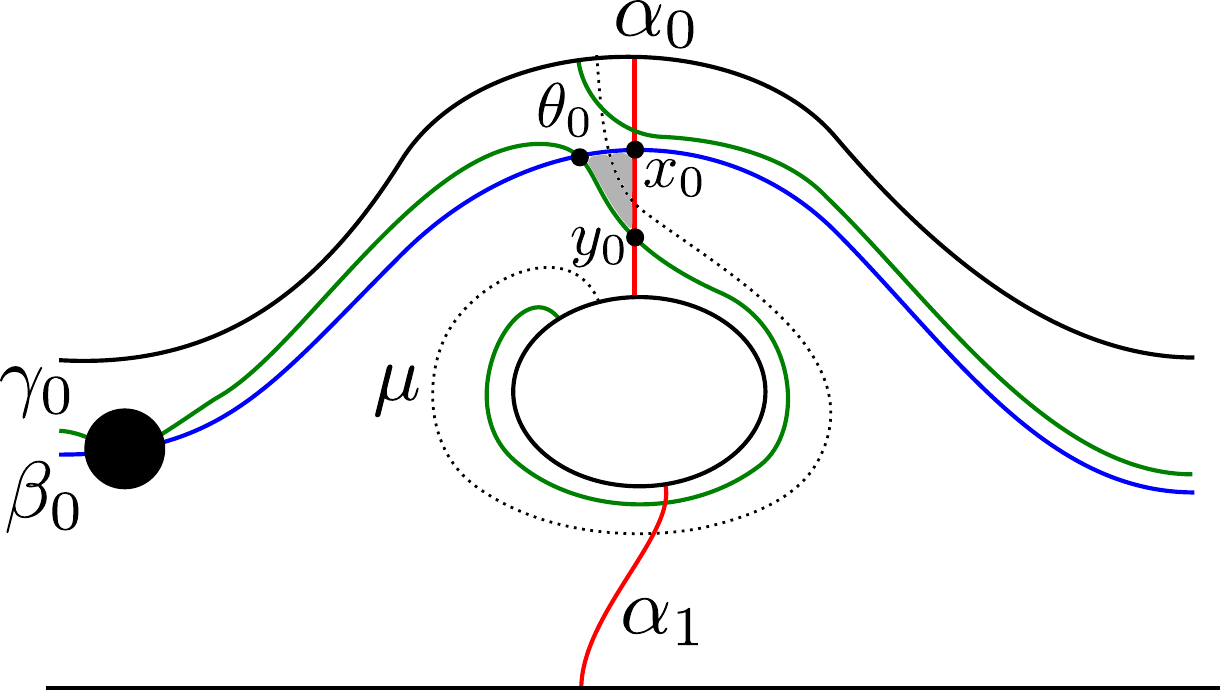}
\end{center}
\caption{In this picture, we represent $F'$ together with a small neighbourhood of $c$ inside $-F'$; the dashed curve represents the meridian $\mu$ for $L'\subset S^3$, and $\gamma_0$ is obtained from $\beta_0$ through a right-handed Dehn twist along $\mu$. We omit the curves $\beta_1$ and $\gamma_1$ to avoid cluttering the picture.}\label{stab_3HD}
\end{figure}

We're now ready to compute the action of the bypass map on $\x_{EH}$; we're going to denote the bypass maps induced by negative stabilisations $\sigma_-$. In order to be able to do a triangle count, we need to perturb the $\beta$-curves to obtain curves $\gamma_1,\dots,\gamma_g$. We choose the perturbations so that $\gamma_i$ has the following two properties:
\begin{itemize}
\item it intersects $\beta_i$ transversely in two points, both inside $F'$ and separated along $\beta_i\cap F'$ by $\alpha_i$ (see Figure \ref{outside_wr1});
\item it intersects $\alpha_i\cap F'$ transversely in a single point $y_i$.
\end{itemize}

The two intersection points of $\beta_i$ and $\gamma_i$ are connected by a bigon $B$ inside $F$. We label them $\theta_i$ and $\theta_i'$ so that this a $B$ connects $\theta_i'$ \emph{to} $\theta_i$. Notice that this is the opposite of the usual convention for triangle counts (see \ref{beta} above). We let $\bfTheta = \{\theta_i\}$.

We're going to do a triangle count in $(\Sigma,\bb,\ba,\bc,D)$; in the notation of \ref{beta} above, the map $\sigma_-$ is induced by $F_- = F_{\beta\alpha\gamma}(\cdot \otimes\bfTheta)$. Let $F_-(\x_{EH}) = \sum_{k=1}^n \y_k$, where all summands are distinct (such a representation exists and is unique up to permutations, since we're working with coefficients in $\F$).

\begin{lemma}\label{positive_is}
For each $k=1,\dots, n$, the intersection point of $\y_k$ along $\alpha_i$ for $i>0$ is $y_i$.
\end{lemma}

\begin{figure}[ht]
\begin{center}
\includegraphics[scale=0.8]{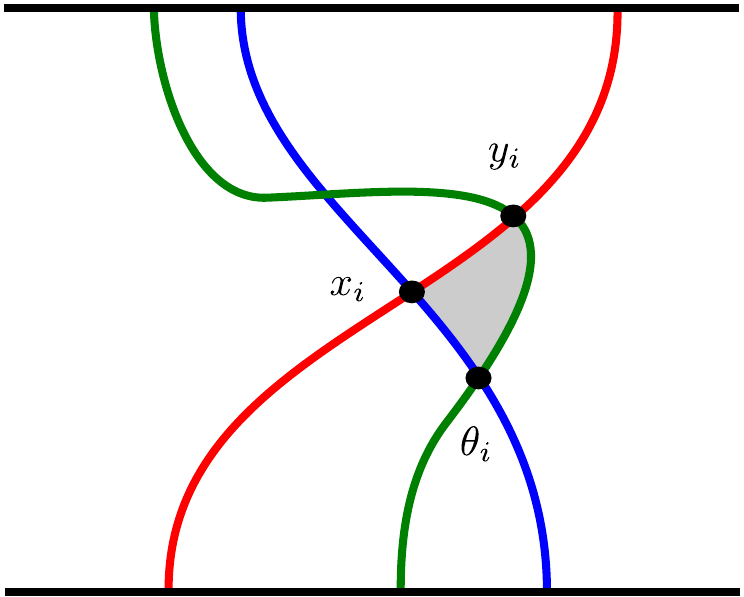}
\end{center}
\caption{This picture represents a neighborhood of $\alpha_i$ in $F'$, for some positive $i$.}\label{outside_wr1}
\end{figure}

\begin{proof}
Let $\psi$ be a holomorphic triangle contributing to the summand $\y_j$ in $\sigma_-(\x) = \sum \y_k$.

Let's consider $F'\subset \Sigma$ in a neighbourhood of $\alpha_i$ containing also $\beta_i$ and $\gamma_i$. The arcs $a_i=\alpha_i\cap F'$ for $i\ge0$ don't disconnect $F'$ by construction; moreover, the arc $\beta_0$ is entirely contained in $F'$ and does not meet any $\alpha_i$ for $i>0$, while $\gamma_0\cap F'$ is made of two arcs that run along $\beta$ except near $\alpha_0$. It follows that the two unbounded regions to the left and right of Figure \ref{outside_wr1} are in fact the same region, which touches the disc $D$. Therefore, the multiplicity of $\D(\psi)$ on this regions is 0.

Since the multiplicity at the left of $x_i$ has to vanish, the corner of $\D(\psi)$ at $x_i$ is acute and is contained in the small triangle, shaded in the picture. Since the multiplicity at the right of $y_i$, too, vanishes, there has to be a corner at $y_i$, to, and in particular the $\alpha_i$-component of $\y_j$ has to be $y_i$. Moreover the domain $\D(\psi)$ has to be a small triangle in the pictured region.
\end{proof}

We now look at the intersection points of $\alpha_0\cap\gamma_0$. Let's call $y_0$ the first intersection point of $\gamma_0$ and $\alpha_0$ we meet when we travel along $\gamma_0$ starting from $D$ and going in the direction of $\theta_0$.

\begin{lemma}\label{xzero}
For each $k=1,\dots, n$, the intersection point of $\y_k$ along $\alpha_0$ is $y_0$.
\end{lemma}

\begin{proof}
The remaining intersection point of $\y_k$ lies on $\alpha_0$ and $\beta_0$, since all other curves already have an intersection point on them. Notice also that there's a small triangle connecting $\theta_0$, $x_0$ and $y_0$, so that $\overline{\y} = \{y_0,\dots,y_g\}$ does in fact appear in the sum. See Figure \ref{stab_3HD}

Consider now a holomorphic triangle $\psi$ and its domain $\D = \D(\psi)$. $\de_\beta\D$ has to be the the short arc connecting $x_0$ and $y_0$ in the handle $H$, since the complement of this arc touches the base-disc $D$. Consider a small push-off $a$ of $\alpha_0$ disjoint from this arc and from $\alpha_0$ itself. Observe that $\de\D$ is nullhomologous and $\de\D \cap a = \de_\gamma\D \cap a$. Therefore, $a$ has to have trivial algebraic intersection with the $\gamma$-boundary of $\D$. Suppose that $\D$ connects $x_0$ with another intersection of $\gamma_0$ with $\alpha_0$: its $\gamma$-boundary $\de_\gamma\D$ is homologous to a linear combination of $\alpha_0$ and the meridian $\mu$, where $\mu$ appears with nonzero multiplicity. In particular, this contradicts the fact that $a$ intersects $\de_\gamma\D$ trivially, since $|a\cap\mu| = 1$.
\end{proof}

In particular, all $\y_k$s are equal, therefore $EH(L'') = \sigma_-(EH(L')) = [\overline{\y}]$.

We want to iterate the procedure, and stabilise $L''$. The bypass we need to attach only modifies $\gamma_0$ by juxtaposition with $\mu$, and in particular Lemma \ref{positive_is} holds in this case as well. Notice also that the only thing we used in proving Lemma \ref{xzero} is that $x_0$ and $y_0$ were the first intersection points of $\alpha_0$ with the arcs $\beta_0$ and $\gamma_0$ respectively, so -- up to notational modifications -- Lemma \ref{xzero} holds for iterations of bypass attachments.

In particular, we've computed the action of $\sigma_-^n$ on $EH(L')$ for every $n\ge 0$.

\vskip 0,3 cm

\textbf{Step 3.} We now slide $\alpha_1$ over $\alpha_0$ to obtain $\alpha_1'$. Recall that $a_i = \alpha_i\cap F'$ intersects the curve $c$ that we used to stabilise the open book $(F,\phi)$ only if $i=0,1$. In particular, $\alpha_1'$ is disjoint from $\mu$ and the only $\alpha$-curve that intersects $\mu$ is $\alpha_0$. Call $\H^{final}$ the Heegaard diagram $(\Sigma, \bb,\ba',D)$.

We're going to compute the action of the map $HS$ induced by this handleslide on the contact invariant. Let $(\Sigma,\bb,\ba,\ba')$ be the triple Heegaard diagram associated to the handleslide, and $\x_{EH}$ be the contact invariant as computed in the previous step and $\y_c$ be the intersection point in $\H^{final}$ that is closest to $\x_{EH}$ (see below for a more precise description).

\begin{lemma}
The handleslide map $HS$ sends $\x_{EH}$ to $\y_c$.
\end{lemma}

\begin{proof}
As above, let $HS(\x_{EH}) = \sum \y_k$ where all summands are distinct.

\begin{figure}[ht]
\begin{center}
\includegraphics[scale=1]{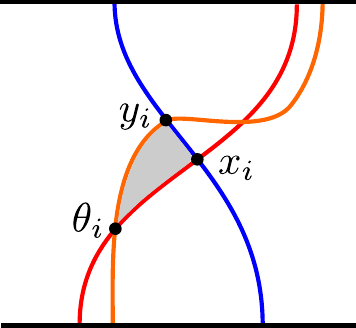}
\end{center}
\caption{This picture represents a neighborhood of $\alpha_i$ in $F'$, for some positive $i$.}\label{outside_wr2}
\end{figure}

On all curves other than $\alpha_0$ and $\alpha_1$ the same argument as in Lemma \ref{positive_is} applies with no modification (see \ref{alpha} for the orientation issues): Figure \ref{outside_wr2} represents what happens locally around $x_i$ and is obtained from Figure \ref{outside_wr1} through a rotation by 180 degrees.

\begin{figure}[ht]
\begin{center}
\includegraphics[scale=0.8]{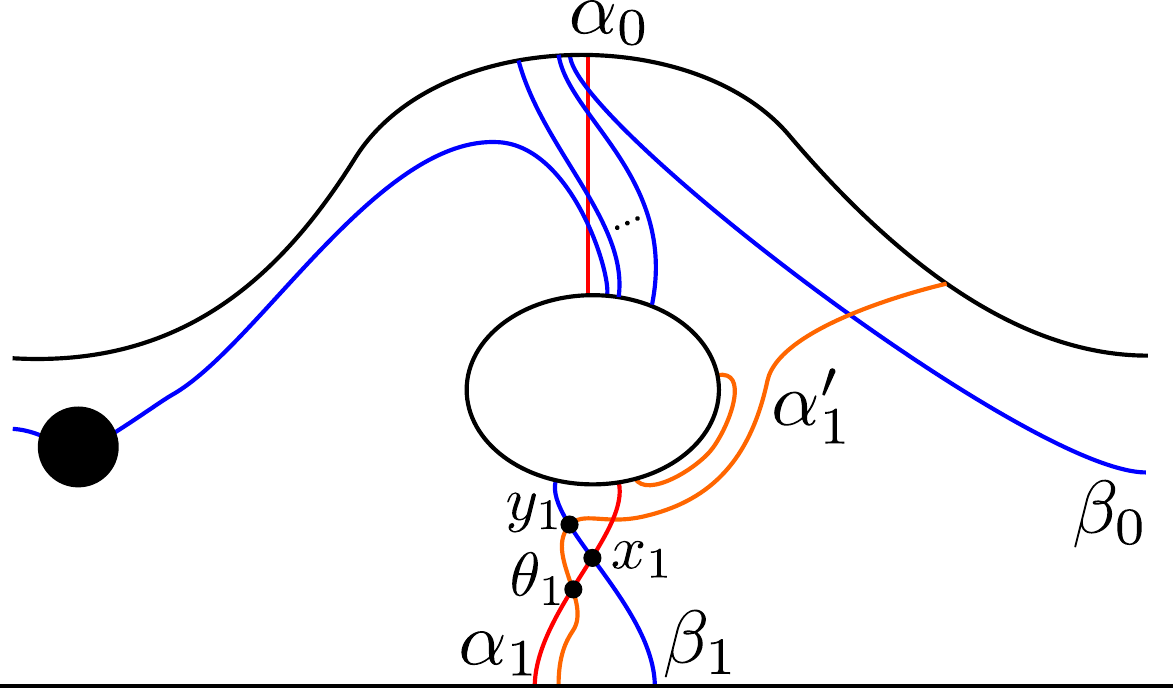}
\end{center}
\caption{This picture represents a neighborhood of the sliding region between $\alpha_0$ and $\alpha_1$ in $F'$}\label{alpha_slide}
\end{figure}

In fact, the same argument applies to the triple $\beta_1, \alpha_1, \alpha_1'$: looking at Figure \ref{alpha_slide}, we see that for every $\y_k$ in the sums the intersection point on $\alpha_1'$ is the intersection of $\alpha_1'$ and $\beta_1$ on $F'$. First of all, there is a small triangle $T_1$ connecting $x_1$ to $y_1$ inside $F'$. To prove that there can be no other domain, we observe that the multiplicity has to vanish in the corner at $x_1$ across from $T_1$, since this region touches the disc $D$. On the other hand, this is enough for the proof of Lemma \ref{positive_is} to work.

Finally, we take care of the intersection point of $\beta_0\cap\alpha_0'$, that is the first intersection point when moving from the $D$ along $\beta_0$, traversing the handle $H$ first. This is similar to the proof of Lemma \ref{xzero} above, and it follows from the same homological considerations.
\end{proof}

\begin{figure}
\begin{center}
\includegraphics{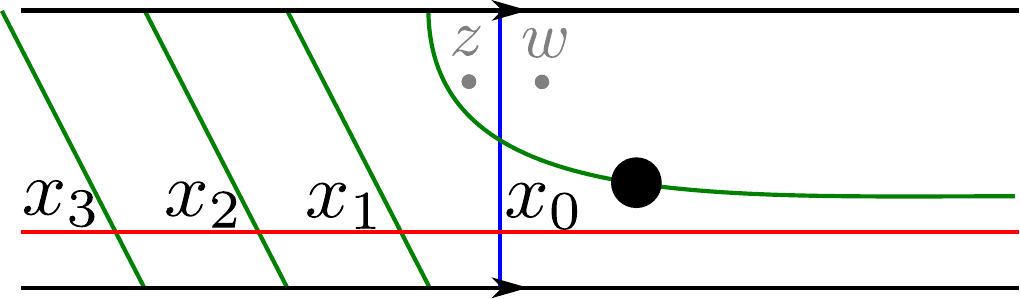}
\caption{The neighbourhood of $\mu_L$ in $\H^{final}$. The twisting is all on one side of $\mu_L$ (the vertical curve). The intersection points on $\alpha'_0$ (the horizontal curve) are labelled $x_0, x_1, \dots$ from right to left. We also put the basepoints $z$ and $w$ (in gray) to represent the knot $K\subset S^3$.}\label{half_winding}
\end{center}
\end{figure}

Observe that a neighbourhood of the meridian $\mu_L$ of $L\subset S^3$ in the diagram looks like \emph{half} of the winding region, as in Figure \ref{half_winding}. Call $x_0$ the intersection point of $\mu_L$ with $\alpha_0'$, and number the intersection points of $\alpha_0'$ with $\beta_0$ as $x_1,x_2,\dots$ according to the order in which we meet them when travelling along $\alpha'_0$ (so that $x_0$ comes first). An easy adaptation of the proof of Theorem 4.1 in \cite{He} shows the following:

\begin{prop}
All generators in $\H^{final}$ with sufficiently large Alexander degree have an intersection point in the winding region.

Moreover, the map $\Phi: SFC_{A\ge N}(-S^3_{L^{stab}})\to CFK^-_{A\ge N'}(S^3,K)$ defined by $\Phi(\{x_n\}\cup\x) = U^n\cdot (\{x_0\}\cup\x)$ induces an isomorphism of chain complexes when $N$ is sufficiently large and $N' = N+\frac{tb(L^{stab})+1}2$.
\end{prop}

In particular, the generator $\y_0$ we've shown to represent $EH(L^{stab})$ is of the form $\{x_1\}\cup\overline\x$, where the generator $\{x_0\}\cup\overline\x\in CFK^-(S^3,K)$ represents $\L^-(L)$. It follows that under map induced at the chain level by $\Phi$ maps $EH(L^{stab})$ to $\L^-(L)$, therefore concluding the proof of Theorem \ref{EH=L-}.
\end{proof}


\begin{thebibliography}{I}\setlength{\itemsep}{-0.5mm}

\small

\bibitem[BVV]{BVV}
J. Baldwin, D. Vela--Vick, V. V\'ertesi: \emph{On the equivalence of Legendrian and transverse invariants in knot Floer homology},  Geom. Topol. \textbf{17} (2013), no. 2, 925--974.

\bibitem[El]{El2}
Y. Eliashberg: \emph{Contact 3-manifolds twenty years since J. Martinet's work}, Ann. Inst. Fourier (Grenoble) \textbf{42} (1992), no. 1-2, 165--192.

\bibitem[EF]{EF}
Y. Eliashberg, M. Fraser: \emph{Topologically trivial Legendrian knots}, J. Symplectic Geom. \textbf{7} (2009), no. 2, 77--127.

\bibitem[Et]{Etn}
J. Etnyre: \emph{Lectures on open book decompositions and contact structures}, in \emph{Floer homology, gauge theory, and low-dimensional topology}, 103--141, Clay Math. Proc. \textbf{5}, Amer. Math. Soc., Providence, RI (2006).

\bibitem[EH]{EH}
J. Etnyre, K. Honda: \emph{Knots and contact geometry I: Torus knots and the figure eight knot}, J. Symplectic Geom. \textbf{1} (2001), no. 1, 63--120.

\bibitem[EV]{EV}
J. Etnyre, J. Van Horn-Morris: \emph{Fibered transverse knots and the Bennequin bound}, Int. Math. Res. Not. IMRN 2011, no. 7, 1483--1509.

\bibitem[EVZ]{EVZ}
J. Etnyre, D. Vela--Vick, R. Zarev: \emph{Bordered sutured Floer homology and invariants of Legendrian knots}, in preparation.

\bibitem[Gi]{Gi}
E. Giroux: \emph{G\'eom\'etrie de contact: de la dimension trois vers les dimensions sup\'erieures}, \emph{Proceedings of the International Congress of Mathematicians, Vol. II},  405--414, Higher Ed. Press, Beijing (2002).

\bibitem[Go]{me}
M. Golla: \emph{Ozsv\'ath-Szab\'o invariants of contact surgeries}, preprint, \url{http://arXiv.org/abs/1201.5286}.

\bibitem[He1]{He}
M. Hedden: \emph{Knot Floer homology of Whitehead doubles}, Geom. Topol. \textbf{11} (2007), 2277--2338.

\bibitem[He2]{He3}
M. Hedden: \emph{Notions of positivity and the Ozsv\'ath-Szab\'o concordance invariant}, J. Knot Theory Ramifications \textbf{19} (2010), no. 5, 617--629.

\bibitem[HP]{HP}
M. Hedden, O. Plamenevskaya: \emph{Dehn surgery, rational open books and knot Floer homology}, Algebr. Geom. Topol. \textbf{13} (2013), no. 3, 1815--1856.

\bibitem[Ho]{Ho}
K. Honda: \emph{On the classification of tight contact structures. I}, Geom. Topol. \textbf{4} (2000), 309--368.

\bibitem[HKM1]{HKM1}
K. Honda, W. Kazez, G. Mati\'c: \emph{The contact invariant in sutured Floer homology}, Invent. Math. \textbf{176} (2009), no. 3, 637--676.

\bibitem[HKM2]{HKM2}
K. Honda, W. Kazez, G. Mati\'c: \emph{Contact structures, sutured Floer homology and TQFT}, preprint, \url{arXiv.org/abs/0807.2431}.

\bibitem[Ju]{Ju}
A. Juh\'asz: \emph{Holomorphic discs and sutured manifolds}, Algebr. Geom. Topol. \textbf{6} (2006), 1429--1457.

\bibitem[Li]{Lip}
R. Lipshitz: \emph{A cylindrical reformulation of Heegaard Floer homology}, Geom. Topol. \textbf{10} (2006), 955--1097.

\bibitem[LOT]{LOT}
R. Lipshitz, P. Ozsv\'ath, D. Thurston: \emph{Bordered Floer homology: invariance and pairing}, preprint, \url{arXiv.org/abs/0810.0687}.

\bibitem[LOSSz]{LOSS}
P. Lisca, P. Ozsv\'ath, A. Stipsicz, Z. Szab\'o: \emph{Heegaard Floer invariants of Legendrian knots in contact three-manifolds}, J. Eur. Math. Soc. \textbf{11} (2009), no. 6, 1307--1363.

\bibitem[On]{On}
S. Onaran: \emph{Invariants of Legendrian knots from open book decompositions}, Int. Math. Res. Not. IMRN \textbf{10} (2010), 1831--1859.

\bibitem[OS]{OSt}
P. Ozsv\'ath, A. Stipsicz: \emph{Contact surgeries and the transverse invariant in knot Floer homology}, J. Inst. Math. Jussieu \textbf{9} (2010), no. 3, 601--632.

\bibitem[OSz1]{OStau}
P. Ozsv\'ath, Z. Szab\'o: \emph{Knot Floer homology and the four-ball genus}, Geom. Topol. \textbf{7} (2003), 615--639.

\bibitem[OSz2]{OSHFK}
P. Ozsv\'ath, Z. Szab\'o: \emph{Holomorphic disks and knot invariants}, Adv. Math. \textbf{186} (2004), no. 1, 58--116.

\bibitem[OSz3]{OSHF}
P. Ozsv\'ath, Z. Szab\'o: \emph{Holomorphic disks and topological invariants for closed three-manifolds}, Ann. of Math. (2) \textbf{159} (2004), no. 3, 1027--1158.

\bibitem[OSz4]{OSPA}
P. Ozsv\'ath, Z. Szab\'o: \emph{Holomorphic disks and three-manifold invariants: properties and applications}, Ann. of Math. (2) \textbf{159} (2004), no. 3, 1159--1245.

\bibitem[OSz5]{OScontact}
P. Ozsv\'ath, Z. Szab\'o: \emph{Heegaard Floer homologies and contact structures}, Duke Math. J. \textbf{129} (2005), no. 1, 39--61.

\bibitem[OSz6]{OSinteger}
P. Ozsv\'ath, Z. Szab\'o: \emph{Knot Floer homology and integer surgeries}. Algebr. Geom. Topol. \textbf{8} (2008), no. 1, 101--153.

\bibitem[OSzT]{OST}
P. Ozsv\'ath, Z. Szab\'o, D. Thurston: \emph{Legendrian knots, transverse knots and combinatorial Floer homology}, Geom. Topol. \textbf{12} (2008), no. 2, 941--980.

\bibitem[Pl]{Pl}
O. Plamenevskaya: \emph{Bounds for the Thurston-Bennequin number from Floer homology}, Algebr. Geom. Topol. \textbf{4} (2004), 399--406.

\bibitem[Ra]{JakeHKM}
J. Rasmussen: \emph{Triangle counts and gluing maps}, in preparation.

\bibitem[SV]{SV}
A. Stipsicz, V. V\'ertesi: \emph{On invariants for Legendrian knots}, Pacific J. Math. \textbf{239} (2009), no. 1, 157--177.

\bibitem[Ve]{VV}
D. Vela--Vick: \emph{On the transverse invariant for bindings of open books}, J. Differential Geom. \textbf{88} (2011), no. 3, 533--552.

\bibitem[Za]{Zarev}
R. Zarev: \emph{Bordered Floer homology for sutured manifolds}, preprint, \url{http://arXiv.org/abs/0908.1106}.
\end{thebibliography}
\end{document}